\documentclass[11pt]{article}

\input diagram.tex
\usepackage{amsmath} % eg. for \begin{equation*}}
\usepackage[mathscr]{eucal}
\usepackage{amssymb} % eg. for \leqslant
\usepackage{theorem}
\usepackage{tikz-cd}
\usetikzlibrary{positioning}

\usepackage{enumerate}

\theoremstyle{change}  % puts numbers IN FRONT of "Theorem"
\newtheorem{theorem}{Theorem}[section] % defines environment "Theorem".
\newtheorem{lemma}[theorem]{Lemma}  % defines environment "Lemma", that
\newtheorem{proposition}[theorem]{Proposition}
\newtheorem{corollary}[theorem]{Corollary}
\theorembodyfont{\rmfamily}  % has the effect that the content of Remark,
\newtheorem{remark}[theorem]{Remark}

\newtheorem{definition}[theorem]{Definition}
\newtheorem{notation}[theorem]{Notation}
\newenvironment{proof}{\noindent{\bf Proof}\ }{\qed\bigskip}

\renewcommand{\le}{\leqslant} % needs amssymb-Paket

\newcommand{\Aut}{\mathrm{Aut}}

\newcommand{\calD}{\mathcal{D}}               
\newcommand{\calE}{\mathcal{E}}
\newcommand{\calF}{\mathcal{F}}

\newcommand{\calI}{\mathcal{I}}

\newcommand{\calP}{\mathcal{P}}

\newcommand{\calS}{\mathcal{S}}

\newcommand{\calT}{\mathcal{T}}

\newcommand{\catfont}{\mathsf}

\newcommand{\Def}{\mathrm{Def}}

\newcommand{\FF}{\mathbb{F}}

\newcommand{\Hom}{\mathrm{Hom}}

\newcommand{\id}{\mathrm{id}}

\newcommand{\Ind}{\mathrm{Ind}}

\newcommand{\Inf}{\mathrm{Inf}}

\newcommand{\Isom}{\mathrm{Iso}}

\newcommand{\lexp}[2]{\setbox0=\hbox{$#2$} \setbox1=\vbox to
                 \ht0{}\,\box1^{#1}\!#2}

\newcommand{\lMod}[1]{\llap{\phantom{|}}_{#1}\catfont{Mod}}

\newcommand{\lset}[1]{\llap{\phantom{|}}_{#1}\catfont{set}}

\newcommand{\NN}{\mathbb{N}}

\newcommand{\Out}{\mathrm{Out}}

\newcommand{\qed}{\nobreak\hfill
                  \vbox{\hrule\hbox{\vrule\hbox to 5pt
                  {\vbox to 8pt{\vfil}\hfil}\vrule}\hrule}}
\newcommand{\Res}{\mathrm{Res}}

\newcommand{\stab}{\mathrm{stab}}

\newcommand{\ZZ}{\mathbb{Z}}

\newcommand{\Br}{\mathrm{Br}}
\newcommand{\Id}{\mathrm{Id}}

\newcommand{\DD}{D^{\Delta}}

\newcommand{\TD}{T^{\Delta}}

\newcommand{\Mv}{M\langle v\rangle}
\newcommand{\Nw}{N\langle w\rangle}

\def\Lu{L\langle u\rangle}

\newcommand{\lcm}{\mathrm{lcm}}
\newcommand{\Dpaircat}{\DD\text{-}\mathsf{pair}}
\newcommand{\pset}[1]{\lset{#1}}

\newcommand{\lbiset}[2]{\llap{\phantom{|}}_{#1}\catfont{set}_{#2}}
\newcommand{\pbiset}[2]{\lbiset{#1}{#2}}
\def\un{\mathbf{1}}
\title{The $\DD$-pair biset category}
\author{Robert Boltje, Serge Bouc and Deniz Y\i lmaz}
\date{}

\begin{document}
\sloppy
\maketitle

\begin{abstract} 

	Let $p$ be a prime number. A $\DD$-pair is a pair consisting of a finite $p$-group and a $p'$-automorphism of the group. In this paper, we introduce diagonal $\DD$-pair bisets and a category whose objects are $\DD$-pairs and whose morphism groups are Grothendieck groups of diagonal $\DD$-pair bisets. Our main result shows that, over suitable coefficient rings, the category of diagonal $p$-permutation functors is equivalent to the category of linear functors on a natural quotient of this new category. In this way, diagonal $p$-permutation functors can be studied through a category built only from finite $p$-groups and their automorphisms of $p'$-order.
\end{abstract}


\section{Introduction}
Let $p$ be a prime number, let $R$ be a unital commutative ring, and $k$ be an algebraically closed field of characteristic $p$. The diagonal $p$-permutation category $Rpp_k^\Delta$ (\cite{BoucYilmaz2020},\cite{BoucYilmaz2022}), has all the finite groups as objects, and morphisms given by $R$-linear extensions of Grothendieck groups of diagonal $p$-permutation bimodules. The category of diagonal $p$-permutation functors over $R$, that is, $R$-linear functors from $Rpp_k^\Delta$ to the category $\lMod{R}$ of all $R$-modules, was introduced in \cite{BoucYilmaz2020} (in the case that $R$ is a field of characteristic 0), and developed in \cite{BoucYilmaz2022}, \cite{BoucYilmaz2024} and \cite{BoucYilmaz2026}. It has allowed for interesting new results in block theory: To each {\em group-block} pair $(G,b)$, that is a pair of a finite group $G$ and a block idempotent $b$ of $kG$, is naturally attached a diagonal $p$-permutation functor $R T^\Delta_{G,b}$ over $R$. This leads to the notion of {\em functorial equivalence} (over $R$) of such group-block pairs (\cite{BoucYilmaz2022}). \par
When $R$ is an algebraically closed field of characteristic zero, the category of diagonal $p$-permutation functors over $R$ is semisimple. Its simple objects $S_{L,u,V}$ are parametrized by equivalence classes of triples $(L,u,V)$ where $(L,u)$ is a $D^\Delta$-pair, that is, a pair consisting of a finite $p$-group $L$ and a $p'$-automorphism $u$ of $L$, and $V$ is a simple $R\Out(L,u)$-module. The multiplicity of a simple functor $S_{L,u,V}$ as a direct summand of $R T^\Delta_{G,b}$ now appears as a new invariant attached to a group block pair $(G,b)$ and a triple $(L,u,V)$. These invariants yield in particular a finiteness theorem (\cite{BoucYilmaz2022}) in the spirit of Donovan's conjecture, saying that for a given $p$-group $D$, there is only a finite number of group block pairs $(G,b)$, where $b$ has defect groups isomorphic to $D$, up to functorial equivalence over $R$. More recently, this led us to an equivalent formulation of Alperin's blockwise weight conjecture in terms of diagonal $p$-permutation functors (\cite{BoltjeBoucYilmaz2026}).\par
Our main goal of the present paper is the introduction of a new $R$-linear category $R\mathcal{P}^\Delta$, see Section~\ref{sec DDelta-pair biset functors}. Its objects are the $D^\Delta$-pairs and its morphisms are the scalar extensions to $R$ of Grothendieck groups of so-called diagonal ``$D^\Delta$-pair bisets'' (also newly introduced here, see Section~\ref{sec Dpair bisets}). This category can be seen as an analogue of the biset category over $R$, where the objects are all finite groups and the morphisms come from bisets for finite groups. The category $R\mathcal{P}^\Delta$ has a particular quotient category $\overline{RP^\Delta}$ with the same objects, and morphisms given by natural equivalence classes of morphisms in $R\mathcal{P}^\Delta$. This quotient category has the following remarkable property: when $R$ is an integral domain in which all $p'$-integers are invertible and which contains all roots of unity of $p'$-order, then the category of $R$-linear functors from $\overline{RP^\Delta}$ to $\lMod{R}$ is equivalent to the category of $R$-linear functors from $Rpp_k^\Delta$ to $\lMod{R}$, see Section~\ref{sec main theorem}. Thus, when considering $R$-linear functors, one can replace the more complicated theory of $p$-permutation $kG$-modules for arbitrary finite groups $G$ with the theory of how $D^\Delta$-pairs $(L,u)$ act on finite sets. Our hope is that, using this equivalence, one can obtain new information on the structure of the diagonal $p$-permutation functors $RT^\Delta_{G,b}$ attached to group-block pairs, only from the knowledge of related information on $p$-groups and their automorphisms of $p'$-order.

\begin{notation}
For any group $G$ and $x\in G$, we denote by $i_x\colon G\to G$ the conjugation map $g\mapsto xgx^{-1}$. We also set $\lexp{x}{g}:=i_x(g)$ and $x^g:= x^{-1}gx$. We denote by $|G|$ the order of $G$ and will sometimes denote the order of a group element $g\in G$ by $|g|$.
\end{notation}

%%%%%%%%%%%%%%%%%%%%% SECTION 2 %%%%%%%%%%%%%%%%%%%%%%%%%%%%%%%%

\section{$\DD$-pairs}

In this section we define the category $\Dpaircat$ of $\DD$-pairs for a fixed prime $p$.

\begin{definition}
	A \emph{$\DD$-pair} $(L,u)$ is a pair consisting of a finite $p$-group $L$ and an element $u\in \Aut(L)$ of $p^\prime$-order, i.e., of order not divisible by $p$. A \emph{morphism} between two $\DD$-pairs $(L,u)$ and $(M,v)$ is a group homomorphism $f: L\to M$ satisfying $v\circ f=f\circ u: L\to M$. Together with the usual composition of group homomorphisms, this defines a category $\Dpaircat$. In fact, if also $g: (M,v)\to (N,w)$ is a morphism of $\DD$-pairs, then the composition $g\circ f$ satisfies $w\circ g\circ f=g\circ v\circ f=g\circ f\circ u$ and is therefore again a morphism. The identity morphism of $(L,u)$ is $\id_L$.
\end{definition}

\begin{remark}
{\rm (a)} Associating to any finite $p$-group $L$ the $\DD$-pair $(L,\id_L)$ extends in an obvious way to a functor from the category of finite $p$-groups and group homomorphisms to the category of $\Dpaircat$. This functor is fully faithful.
	
\smallskip
{\rm (b)} The category $\Dpaircat$ has finite products. For $\DD$-pairs $(L,u)$ and $(M,v)$ we set $(L,u)\times (M,v):=(L\times M,u\times v)$. Together with the usual projection maps $L\times M\to L$ and $L\times M\to M$, this is a categorical product of $(L,u)$ and $(M,v)$.
	
\smallskip
(c) For any $\DD$-pair one can consider the semidirect product $\Lu=L\rtimes\langle u\rangle$ whose elements will be written as $lu^i$ with $l\in L$ and $i\in \ZZ$. One has the commutation rule $u^i l=u^i(l)u^i$ in $L\langle u\rangle$. Note that $L$ is the unique Sylow $p$-subgroup of $L\langle u\rangle$, hence the set of all $p$-elements of $\Lu$.
\end{remark}

\begin{lemma}
	Let $(L,u)$ and $(M,v)$ be $\DD$-pairs and let $f: L\to M$ be a surjective group homomorphism satisfying $v\circ f=f\circ u$. Then there exists a unique extension $F: \Lu \to \Mv$ of $f$ with $F(u)=v$.
\end{lemma}

\begin{proof}
	The uniqueness of $F$ is clear, since $L$ and $u$ generate $L\langle u\rangle$. To prove the existence, first note that if $i\in \ZZ$ is such that $u^i=\id_L$ then also $v^i=\id_M$. Indeed, $u^i=\id_L$ implies $v^i\circ f=f\circ u^i=f$ and since $f$ is surjective, this implies $v^i=\id$. This now implies that the function $F: L\langle u\rangle\to \Mv$, $lu^i\mapsto f(l)v^i$, where $l\in L$ and $i\in \ZZ$, is well-defined. The function $F$ clearly extends $f$ and it remains to be shown that $F$ is a group homomorphism. So let $l,l'\in L$ and $i,j\in \ZZ$. Then
	\[
	F(lu^il'u^j)=F(lu^i(l')u^{i+j})=f(lu^i(l'))v^{i+j}
	\]
	and also
	\[
	F(lu^i)F(l'u^j)=f(l)v^if(l')v^j=f(l)v^i(f(l'))v^{i+j}=f(l)f(u^i(l'))v^{i+j}=f(lu^i(l'))v^{i+j}.
	\]
\end{proof}

\begin{corollary}
	Let $(L,u)$ and $(M,v)$ be $\DD$-pairs. Then the following are equivalent:
	
	\smallskip
	{\rm (i)} $(L,u)\cong (M,v)\in \Dpaircat$.
	
	\smallskip
	{\rm (ii)} There exists a group isomorphism $f: L\to M$ satisfying $v\circ f=f\circ u$.
	
	\smallskip
	{\rm (iii)} There exists a group isomorphism $F: L\langle u\rangle\to M\langle v\rangle$ satisfying $F(u)=v$.
		
	\smallskip
	{\rm (iv)} There exists a group isomorphism $F: L\langle u\rangle\to M\langle v\rangle$ such that $F(u)$ is $M\langle v\rangle$-conjugate to $v$.
\end{corollary}

\begin{proof}
	{\rm (i)} and {\rm (ii)} are equivalent by definition, since if $f$ is as in {\rm (ii)} then also $f^{-1}$ is a morphism of $\DD$-pairs. By Lemma 2.3, {\rm (ii)} implies {\rm (iii)}. Moreover, {\rm (iii)} obviously implies {\rm (iv)}. To prove that {\rm (iv)} implies {\rm (ii)}, let $F: L\langle u\rangle\to M\langle v\rangle$ be a group isomorphism and $m\in M$ such that $F(u)=v^m$. Then $F|_L$ is a group isomorphism $L\to M$ between the respective unique Sylow $p$-subgroups of $L\langle u\rangle$ and $M\langle v\rangle$. Moreover, for any $l\in L$, the group isomorphism $f:=i_m\circ F|_L: L\to M$ satisfies
	\begin{align*}
	(v\circ f)(l)&=v(i_m(F(l)))=v\cdot i_m(F(l))\cdot v^{-1}=i_m\bigl(v^m\cdot F(l)\cdot (v^{-1})^m\bigr)\\&
	=i_m\bigl(F(u)\cdot F(l)\cdot F(u^{-1})\bigr)=i_m(F(u\cdot l\cdot u^{-1}))=i_m(F(u(l)))=f(u(l)),
	\end{align*}
	so that $f$ satisfies {\rm (ii)}. This completes the proof.
\end{proof}

\begin{remark}\label{rem subpairs}
	Let $(L,u)$ be a $\DD$-pair.
	
	\smallskip
	{\rm (a)} Suppose that $X\le L$ is $u$-invariant, i.e., $u(X)=X$. Then the restriction $u':=u|_X$ of $u$ to $X$ is an automorphism of $X$ of $p'$-order, so that $(X,u')$ is a $\DD$-pair. In this case we call $(X,u')$ a {\em $\DD$-subpair} of $(L,u)$. Note that the inclusion map $\iota: X\to L$ is a morphism from $(X,u')$ to $(L,u)$ in the category $\Dpaircat$.
	
	\smallskip
	{\rm (b)} Suppose that $X\unlhd L$ is $u$-invariant. Then there exists a unique element $\bar u\in \Aut(L/X)$ such that $\bar u\circ \pi=\pi\circ u$, where $\pi: L\to L/X$ denotes the canonical projection. In this case we call $(L/X,\bar u)$ a {\em quotient $\DD$-pair} of $(L,u)$. Note that the map $\pi$ is a morphism from $(L,u)$ to $(L/X,\bar u)$ in $\Dpaircat$.
\end{remark}

%%%%%%%%%%%%%%%%%%%%%%%% SECTION 3 %%%%%%%%%%%%%%%%%%%%%%%%%%%%%%

\section{$\DD$-pairs acting on sets}

In this section we define, for a $\DD$-pair $(L,u)$, the category
$\pset{(L,u)}$ of $(L,u)$-sets. Recall that $L\langle u\rangle=L\rtimes \langle u\rangle$. If $u=\id_L$, this specializes to the usual category of finite $L$-sets.

\begin{definition}\label{def (L,u)-set}
	Let $(L,u)$ be a $\DD$-pair. An \emph{$(L,u)$-set} is a finite $L$-set $\Omega$ equipped with an action of $\langle u\rangle$, denoted by $\omega\mapsto u\cdot \omega$ with the following properties:
	
	\smallskip
	{\rm (i)} If $l\in L$ and $\omega\in \Omega$, then
		\begin{align*}
			u\cdot (l\omega)=u(l)(u\cdot \omega)\,.
		\end{align*}
		
	\smallskip
	{\rm (ii)} Every $L$-orbit on $\Omega$ is invariant by $u$, that is,
		\begin{align*}
			\forall \omega\in \Omega,\ u\cdot \omega\in L\omega\,.
		\end{align*}  
\end{definition}

\begin{remark}\label{rem (L,u)-sets}
(a) Condition~{\rm (i)} is equivalent to saying that $\Omega$ is an $\Lu$-set. Thus, an $(L,u)$-set is nothing else than an $\Lu$-set with the property that its $\Lu$-orbits and $L$-orbits coincide. This implies that a disjoint union of finitely many $\Lu$-sets $\Omega_i$, $i\in I$, is an $(L,u)$-set if and only if every $\Omega_i$, $i\in I$, is an $(L,u)$-set. In particular, the $\Lu$-orbits of an $(L,u)$-set are again $(L,u)$-sets.

\smallskip
(b) If $\Omega$ is an $(L,u)$-set then, by Glauberman's Lemma (\cite{Isaacs} Lemma~3.24),
Condition~(ii) implies that each $L$-orbit contains a point fixed by~$u$ and that the $\langle u\rangle$-fixed points in an $L$-orbit form a single $L^u$-orbit. Here, we denote by $L^u$ the set of $u$-fixed points on $L$, that is 
$$L^u:=\{l\in L\mid u(l)=l\}.$$
Note that $L^u=L\cap C_{\Lu}(u)$. We will denote the set of $u$-fixed points of $\Omega$ by $\Omega^u$. Thus, the sets of $L^u$-orbits of $\Omega^u$ and $\Lu$-orbits of $\Omega$ are in canonical bijection, by mapping the $L^u$-orbits of $\omega\in \Omega^u$ to its $\Lu$-orbit (or equivalently $L$-orbit).
\end{remark}

\begin{definition}
	A \emph{morphism} of $(L,u)$-sets is just a morphism of $\Lu$-sets. This way one obtains a category $\pset{(L,u)}$ of $(L,u)$-sets, which is the full subcategory $\pset{L\langle u\rangle}$ consisting of $(L,u)$-sets. By Part~(a), disjoint unions together with their inclusions form coproducts in the category $\pset{(L,u)}$.
\end{definition}

\begin{definition}
An $(L,u)$-set $\Omega$ is called \textit{transitive} if it cannot be written as a disjoint union of two non-empty $(L,u)$-sets. By Remark~\ref{rem (L,u)-sets}(a), this just means that $\Omega$ is transitive as $\Lu$-set.
\end{definition}

The goal for the remainder of this section is to obtain a better understanding of transitive $(L,u)$-sets.

\begin{lemma}\label{lem sometransitives}
	Let $(L,u)$ be a $\DD$-pair and let $X\leq L$. Consider the ordinary
	transitive $L$-set $L/X$. 
	One obtains a well-defined action of $\langle u\rangle$ on $L/X$ by setting
	\[
	u^i\cdot lX:=u^i(l)X
	\qquad (l\in L,\ i\in\ZZ)
	\]
	if and only if $X$ is $u$-invariant. In this case this definition endows $L/X$ with a transitive $(L,u)$-set structure.
\end{lemma}

\begin{proof}
	We first show that the action of $\langle u\rangle$ on $L/X$ is well-defined if and only if $X$ is $u$-invariant. Suppose first that the action of $\langle u\rangle$ is well-defined. Then for $x\in X$, we have
	\[
	u(x)X=u\cdot xX=u\cdot X=X\,.
	\]
	Thus $u(x)\in X$ and hence $u(X)=X$. Conversely, suppose that $X$ is $u$-invariant. If $lX=l'X$, then $l^{-1}l'\in X$, and so $u^i(l^{-1})u^i(l')=u^i(l^{-1}l')\in X$ for any $i\in\ZZ$. Thus,
	\[
	u^i(l)X=u^i(l')X,
	\]
	so the action of $\langle u\rangle$ is well-defined.
	
	Now, for $l_1,l_2\in L$, we have
	\[
	u\cdot(l_1l_2X)=u(l_1l_2)X=u(l_1)u(l_2)X=u(l_1)(u\cdot l_2X),
	\]
	so the $L$-action and the $\langle u\rangle$-action are compatible. Also, for any $l\in L$,
	\[
	%u\cdot (lX)=u(l)X\in L(lX)=L/X\,.
	u\cdot (lX)=u(l)X\in L/X\,.
	\]
	Therefore $L/X$ is an $(L,u)$-set. Finally, $L/X$ is clearly a transitive $(L,u)$-set.
\end{proof}

\begin{lemma}\label{lem transitivedescription}
	Every transitive $(L,u)$-set is isomorphic to $L/X$ (with its structure as in Lemma~\ref{lem sometransitives}) for some
	$u$-invariant subgroup $X$ of $L$.
\end{lemma}
\begin{proof}
	Let $\Omega$ be a transitive $(L,u)$-set. Then $\Omega$ is transitive as an $L$-set. By Glauberman's Lemma $\Omega$ contains a point $\omega$ fixed by $u$. Set 
	\[
	X=\mathrm{Stab}_L(\omega).
	\]
	Then $X$ is $u$-invariant. Indeed, if $x\in X$, then
	\[
	u(x)\omega=u(x)(u\cdot\omega)=u\cdot(x\omega)=u\cdot\omega=\omega\,.
	\]
	Hence $u(x)\in X$.
	
	Now, the usual $L$-set isomorphism
	\[
	f:L/X\longrightarrow \Omega,\qquad lX\mapsto l\omega
	\]
	is also $\langle u\rangle$-equivariant, since
	\[
	f(u\cdot (lX))=f(u(l)X)=u(l)f(X)=u(l)\omega=u(l)(u\cdot\omega)=u\cdot (l\omega)=u\cdot f(lX)\,.
	\]
	Hence $\Omega\cong L/X$ as $(L,u)$-sets.
\end{proof}

\begin{lemma}\label{lem isomorphictransitives}
	Let $X,Y\leq L$ be $u$-invariant subgroups. Then $L/X$ and $L/Y$ are
	isomorphic as $(L,u)$-sets if and only if there exists $a\in L^u$ such that
	\[
	X=aYa^{-1}.
	\]
\end{lemma}

\begin{proof}
	If $a\in L^u$ with $X=aYa^{-1}$, then the map
	\[
	f:L/X\longrightarrow L/Y,\qquad lX\mapsto laY
	\]
	is an isomorphism of $L$-sets. It is also compatible with the action of $u$. Indeed, since $u(a)=a$, we have
	\[
	f(u\cdot (lX))=f(u(l)X)=u(l)aY=u(l)u(a)Y=u(la)Y=u\cdot laY=u\cdot f(lX)\,.
	\]
	Conversely, suppose that
	\[
	\varphi: L/X\longrightarrow L/Y
	\]
	is an isomorphism of $(L,u)$-sets. In particular, $\varphi$ is an isomorphism of transitive $L$-sets. Hence there exists $b\in L$ such that
	\[
	X=bYb^{-1}\qquad\text{and}\quad \varphi(lX)=lbY\,.
	\]
	Since $\varphi$ is $\langle u\rangle$-equivariant and the element $X\in L/X$ is $u$-fixed, it follows that $bY=\varphi(X)$ is also $u$-fixed. The element $Y\in L/Y$ is also $u$-fixed. By Glauberman's Lemma, the $u$-fixed elements of the transitive $L$-set $L/Y$ form a transitive $L^u$-set. Hence there exists $a\in L^u$ such that
	\[
	bY=aY\,.
	\]
	Therefore
	\[
	X=bYb^{-1}=aYa^{-1}.
	\]
	This proves the claim.
\end{proof}
\begin{corollary}\label{cor isomclassesoftransitive}
	The isomorphism classes of transitive $(L,u)$-sets are in bijection with
	the $L^u$-conjugacy classes of $u$-invariant subgroups of $L$.
\end{corollary}

\begin{proof}
	By Lemma~\ref{lem transitivedescription}, every transitive $(L,u)$-set is isomorphic to $L/X$ for some $u$-invariant subgroup $X\leq L$, and by Lemma~\ref{lem isomorphictransitives} two such transitive $(L,u)$-sets $L/X$ and $L/Y$ are isomorphic if and only if $X$ and $Y$ are conjugate by an element of $L^u$.
\end{proof}

%%%%%%%%%%%%%%% SECTION 4 %%%%%%%%%%%%%%%%%%%%%%%%%%%%%%%

\section{The $((M,v),(L,u))$-bisets}\label{sec Dpair bisets}

\begin{definition}
	Let $(L,u)$ and $(M,v)$ be $\DD$-pairs. An $((M,v),(L,u))\text{-biset}$
	is a finite $(M,L)$-biset $\Omega$ equipped with an action of $\langle (v,u)\rangle$, denoted by
	\[
	\omega\mapsto (v,u)\cdot \omega,
	\]
	with the following properties:
	
	\smallskip
	{\rm (i)} If $m\in M$, $l\in L$, and $\omega\in \Omega$, then
		\[
		(v,u)\cdot (m\omega l)=v(m)\bigl((v,u)\cdot \omega\bigr)u(l).
		\]
	
	\smallskip
	{\rm (ii)} Every $(M,L)$-orbit on $\Omega$ is invariant under $(v,u)$, that is,
		\[
		\forall \omega\in \Omega,\qquad (v,u)\cdot \omega\in M\omega L.
		\]
\end{definition}

\begin{remark}\label{rem DDelta bisets}
With the usual way of translating a right action to a left action, an $((M,v),(L,u))$-biset is the same thing as an $(M\times L, (v, u))$-set, as in Definition~\ref{def (L,u)-set}, where $(v,u)$ denotes the automorphism of $M\times L$ given by $(m,l)\mapsto (v(m),u(l))$. Equivalently, an $((M,v),(L,u))$-biset is an $(M\times L)\rtimes \langle (v,u)\rangle$-set with the property that every $M\times L$-orbit is invariant under $(v,u)$.  We therefore have a category
\begin{equation*}
   \lbiset{(M,v)}{(L,u)}:=\pset{(M\times L, (v,u))}
\end{equation*}
of $((M,v),(L,u))$-bisets, which can be considered as the full subcategory of $(M\times L)\langle(v,u)\rangle$-sets whose objects satisfy that  
their $M\times L$-orbits are invariant under $(v,u)$, or equivalently that their $(M\times L)$-orbits and $(M\times L)\langle (v,u)\rangle$-orbits coincide (see Remark~\ref{rem (L,u)-sets}(a)). 

\smallskip
Glauberman's Lemma applies in the same way as in Remark~\ref{rem (L,u)-sets}(b) with
$(M\times L)^{(v,u)}=M^v\times L^u$.
\end{remark}

\begin{remark}\label{Un 1}
Let $(\un,1)$ denote the {\em trivial} $\DD$-pair, consisting of the trivial $p$-group $\un$, and its unique (identity) automorphism. Then, for any $\DD$-pair $(L,u)$, the $\big((L,u),(\un,1)\big)$-bisets identify trivially with the $(L,u)$-sets. This identification is in fact an isomorphism of categories between $\lbiset{(L,u)}{(\un,1)}$ and $\pset{(L,u)}$.
\end{remark}

The following
Proposition follows immediately from Lemma~\ref{lem sometransitives}, Lemma~\ref{lem isomorphictransitives},
and Corollary~\ref{cor isomclassesoftransitive}.

\begin{proposition}\label{prop transitivebisets}
Let $(L,u)$ and $(M,v)$ be $\DD$-pairs. Up to isomorphism, the transitive $((M,v),(L,u))$-bisets are precisely the transitive $(M\times L)\langle(v,u)\rangle$-sets of the form $(M\times L)/X$, where $X\le M\times L$ is $(v,u)$-invariant. 

If $X$ and $Y$ are $(v,u)$-invariant subgroups of $M\times L$ then $(M\times L)/X$ and $(M\times L)/Y$ are isomorphic as $((M,v),(L,u))$-bisets if and only if $X$ and $Y$ are conjugate by an element in $M^v\times L^u$. In particular, the isomorphism classes of transitive $((M,v),(L,u))$-bisets are in bijection with the $(M^v\times L^u)$-conjugacy classes of $(v,u)$-invariant subgroups of $M\times L$. 
\end{proposition}

 In the following definition we use the {\em composition} of two bisets. See \cite[Definition~2.3.11]{Bouc2010a} for the definition of this construction and related notations.

\begin{definition}\label{composition}
	Let $(L,u)$, $(M,v)$ and $(N,w)$ be $\DD$-pairs. Let $\Omega$ be an
	$((M,v),(L,u))$-biset and let $\Gamma$ be an $((N,w),(M,v))$-biset.	We denote by
	\[
	\Omega^{(v,u)}=\{\omega\in \Omega\mid (v,u)\cdot \omega=\omega\}
	\]
	the set of $\langle(v,u)\rangle$-fixed points on $\Omega$. Then $\Omega^{(v,u)}$ is naturally an $(M^v,L^u)$-biset. Indeed, for $(m,l)\in M^v\times L^u$ and $\omega\in \Omega^{(v,u)}$, one has
	\[
	(v,u)\cdot (m\omega l)=v(m)((v,u)\cdot \omega)u(l)=m\omega l\,.
	\]
	Similarly, $\Gamma^{(w,v)}$ is naturally an $(N^w,M^v)$-biset. Hence we may form the biset product
	\[
	\Gamma^{(w,v)}\times_{M^v} \Omega^{(v,u)}.
	\]
	There is a natural map
	\[
	\theta\colon
	\Gamma^{(w,v)}\times_{M^v} \Omega^{(v,u)}\longrightarrow \Gamma \times_M \Omega, \qquad (\gamma,{}_{M^v}\omega)\mapsto (\gamma,{}_M \omega).
	\]
	We define the $(N,L)$-biset $\Gamma\times_{M,v}\Omega$ as the %\textit{composition} of $Y$ and $X$ to be the 
	$(N,L)$-subbiset
	\[
	\Gamma\times_{M,v}\Omega:=N\cdot \operatorname{Im}(\theta)\cdot L \subseteq \Gamma\times_M \Omega
	\]
	of $\Gamma\times_M \Omega$. Equivalently, $\Gamma\times_{M,v}\Omega$ is the $(N,L)$-subbiset of $\Gamma\times_M \Omega$
	generated by the elements $(\gamma,{}_M \omega)$ with
	\[
	\gamma\in \Gamma^{(w,v)} \qquad\text{and}\qquad \omega\in \Omega^{(v,u)}.
	\]
\end{definition}

\begin{lemma}\label{composition is functorial}
	With the notation from the previous definition, the $(N,L)$-biset $\Gamma\times_{M,v}\Omega$ is an
	$((N,w),(L,u))$-biset with the $\langle(w,u)\rangle$-action given by
	$$(w,u)\cdot (\gamma,{}_M \omega)=((w,v)\cdot \gamma,{}_M (v,u)\cdot \omega)\,,$$
	for any element $(\gamma,_M \omega)\in \Gamma\times_{M,v}\Omega$. Moreover, this defines a functor
	$$\lbiset{(N,w)}{(M,v)}\times\lbiset{(M,v)}{(L,u)}\to \lbiset{(N,w)}{(L,u)}\,, \quad (\Gamma,\Omega)\mapsto \Gamma\times_{M,v} \Omega\,,$$
	which preserves coproducts in both arguments.
\end{lemma}

\begin{proof}
First we observe that the function
\begin{equation*}
   \sigma\colon \Gamma\times_M \Omega \to \Gamma\times_M \Omega\,, \quad 
   (\gamma,_M \omega)\mapsto \bigl((w,v)\cdot \gamma,_M (v,u)\cdot \omega\bigr)\,,
\end{equation*}
 is well defined. Indeed, for $m\in M$, one has 
\begin{align*} 
   ((w,v)\cdot(\gamma m),{}_M (v,u)\cdot(m^{-1}\omega)) & = (((w,v)\cdot \gamma)v(m), {}_M v(m^{-1})((v,u)\cdot \omega))\\
   & = ((w,v)\cdot \gamma,{}_M (v,u)\cdot \omega)\,.
 \end{align*}
 Moreover, 
 if $\gamma\in \Gamma^{(w,v)}$ and $\omega\in \Omega^{(v,u)}$, then
	\[	\sigma(\gamma,{}_M \omega)=(\gamma,{}_M \omega)\,.
	\]
	Thus, every element of $\operatorname{Im}(\theta)$ is fixed by $\sigma$. It follows that the $(N,L)$-subbiset generated by $\operatorname{Im}(\theta)$ is stable under $\sigma$, because
\begin{align}\label{eqn sigma}
   \sigma (n\gamma,_M \omega l) & = \bigl( (w,v)\cdot (n\gamma),_M (v,u)\cdot (\omega l) \bigr) = \bigl( w(n) ((w,v)\cdot \gamma)),_M ((v,u)\cdot \omega) u(l)\bigr)\\
   \notag & = \bigl( w(n) \gamma,_M \omega u(l)\bigr) \,,
\end{align}
for any $n\in N$, $l\in L$, $\gamma\in \Gamma^{(w,v)}$ and $\omega\in \Omega^{(v,u)}$.
Therefore, $\sigma$ restricts to a map $\Gamma\times_{M,v}\Omega\to \Gamma\times_{M,v} \Omega$. Moreover, if $(w,u)^i=1$ then the last computation shows that $\sigma^i(\gamma,_M\omega)= (\gamma,_M\omega)$ for any $(\gamma,_M\omega)\in \Gamma\times_{M,v} \Omega$.
Thus, setting
\begin{equation*}
   (w,u)\cdot (\gamma,_M \omega):= \sigma(\gamma,_M \omega) = \bigl( (w,v)\cdot \gamma,_M (v,u)\cdot \omega\bigr)\,,
\end{equation*}
for $(\gamma,_M \omega)\in \Gamma\times_{M,v} \Omega$, defines an action of $\langle (w,u)\rangle$ on $\Gamma\times_{M,v} \Omega$. Equation~(\ref{eqn sigma}) also shows that 
\begin{equation*}
   (w,u)\cdot \bigl(n(\gamma,_M \omega)l\bigr) = w(n)(\gamma,_M \omega)u(l)\,,
\end{equation*}
for all $(\gamma,_M \omega)\in \Gamma\times_{M,v} \Omega$ and $(n,l)\in N\times L$,
so that $\Gamma\times_{M,v} \Omega$  becomes an $(N\times L)\rtimes\langle(w,u)\rangle$-set.
Clearly, every $(N,L)$-orbit in $\Gamma\times_{M,v}\Omega$ contains an element of $\operatorname{Im}(\theta)$, hence contains an element fixed by $(w,u)$. Therefore, every $(N,L)$-orbit is invariant under $(w,u)$ and $\Gamma\times_{M,v}\Omega$ is an $((N,w),(L,u))$-biset.
	
If
        \[
	\beta\colon \Gamma\to \Gamma'
	\]
	is a morphism of $((N,w),(M,v))$-bisets and
	\[
	\alpha\colon \Omega\to \Omega'
	\]
	is a morphism of $((M,v),(L,u))$-bisets, then the induced map
\[
\beta\times_M\alpha\colon \Gamma\times_M \Omega\longrightarrow \Gamma'\times_M \Omega', \qquad (\gamma,{}_M \omega)\mapsto (\beta(\gamma),{}_M \alpha(\omega)),
\]
restricts to a map
\[
\Gamma\times_{M,v}\Omega\longrightarrow \Gamma'\times_{M,v}\Omega'.
\]
Indeed, $\beta$ sends $\Gamma^{(w,v)}$ into $(\Gamma')^{(w,v)}$, and $\alpha$ sends
$\Omega^{(v,u)}$ into $(\Omega')^{(v,u)}$. Thus the image of the generating subset of
$\Gamma\times_{M,v}\Omega$ is contained in the generating subset of
$\Gamma'\times_{M,v}\Omega'$. Hence composition is functorial in both variables.
Finally, it is straightforward to show that this construction preserves coproducts in both arguments.
\end{proof}

We call the above $((N,w),(L,u))$-biset $\Gamma\times_{M,v} \Omega$ the {\em composition} of the $((N,w),(M,v))$-biset $\Gamma$ and the $((M,v),(L,u))$-biset $\Omega$.

\begin{proposition}\label{composition is associative}
	Let $(K,t)$, $(N,w)$, $(M,v)$ and $(L,u)$ be $\DD$-pairs. Let $\Theta$ be a $((K,t),(N,w))$-biset, let $\Gamma$ be a $((N,w),(M,v))$-biset, and let $\Omega$ be a $((M,v),(L,u))$-biset. Then there is a natural isomorphism of
	$((K,t),(L,u))$-bisets
	\[
	(\Theta\times_{N,w}\Gamma)\times_{M,v}\Omega \cong \Theta\times_{N,w}(\Gamma\times_{M,v}\Omega).
	\]
\end{proposition}

\begin{proof}
	We claim that the natural associativity isomorphism 
\begin{equation*}
   (\Theta\times_N \Gamma)\times_M \Omega\cong \Theta\times_N (\Gamma\times_M \Omega)\,,\quad (\theta,_N \gamma),_M \omega) 
   \mapsto (\theta,_N(\gamma,_M\omega))\,,
\end{equation*}
restricts to an isomorphism between
	\[
	(\Theta\times_{N,w}\Gamma)\times_{M,v}\Omega\qquad\text{and}\qquad \Theta\times_{N,w}(\Gamma\times_{M,v}\Omega)
	\]
	which is also $\langle(t,u)\rangle$-equivariant.
	This will prove the result. 
	
Let $S_{\Theta\Gamma}$ denote the image of
	\[
	\Theta^{(t,w)}\times_{N^w}\Gamma^{(w,v)}
	\longrightarrow
	\Theta \times_N\Gamma.
	\]
	By definition,
	\[
	\Theta\times_{N,w}\Gamma=K S_{\Theta\Gamma} M.
	\]
	We need to understand the $(t,v)$-fixed points of this composite. Since $\Theta\times_{N,w}\Gamma$ is generated, as a $(K,M)$-biset, by the $(t,v)$-fixed set $S_{\Theta\Gamma}$, every $(K,M)$-orbit in $\Theta\times_{N,w}\Gamma$ contains a $(t,v)$-fixed point. By Glauberman's Lemma, the fixed points in such an orbit form a single orbit under $K^t\times M^v$. Hence,
	\[
	(\Theta\times_{N,w}\Gamma)^{(t,v)}=K^t S_{\Theta\Gamma} M^v.
	\]
	Therefore
	\[
	(\Theta\times_{N,w}\Gamma)\times_{M,v}\Omega
	\]
	is the $(K,L)$-subbiset of $(\Theta\times_N\Gamma)\times_M\Omega$ generated by the image of
	\[
	(K^t S_{\Theta\Gamma} M^v)\times_{M^v}\Omega^{(v,u)}.
	\]
	This subbiset is generated by the elements 
	$((\theta,{}_N \gamma),{}_M \omega)$
	with
	\[
	\theta\in \Theta^{(t,w)},\qquad \gamma\in \Gamma^{(w,v)},\qquad \omega\in \Omega^{(v,u)}.
	\]
	
	The same argument on the other side gives
	\[
	(\Gamma\times_{M,v}\Omega)^{(w,u)}=N^w S_{\Gamma\Omega} L^u,
	\]
	where $S_{\Gamma\Omega}$ is the image of
	\[
	\Gamma^{(w,v)}\times_{M^v}\Omega^{(v,u)} \longrightarrow \Gamma\times_M\Omega.
	\]
	Hence,
	\[
	\Theta\times_{N,w}(\Gamma\times_{M,v}\Omega)
	\]
	is also the $(K,L)$-subbiset of 
	$\Theta\times_N(\Gamma\times_M\Omega)$ 
	generated by the elements
	$(\theta,{}_N (\gamma,{}_M \omega))$
	with $\theta$, $\gamma$, and $\omega$ fixed by $(t,w)$, $(w,v)$, and $(v,u)$ respectively.
	
	Thus, the usual associativity isomorphism of biset compositions restricts to an isomorphism
	\[
	(\Theta\times_{N,w}\Gamma)\times_{M,v}\Omega \cong \Theta\times_{N,w}(\Gamma\times_{M,v}\Omega),
	\]
	and this isomorphism is compatible with the action of $\langle(t,u)\rangle$. Therefore it is an isomorphism of $((K,t),(L,u))$-bisets.
\end{proof}

\begin{proposition}\label{identity}
	For each $\DD$-pair $(L,u)$, the regular $(L,L)$-biset $L$, equipped with the action of
	$\langle(u,u)\rangle$ given by
	\[
	(u^i,u^i)\cdot l=u^i(l),
	\]
	is an identity element for the composition of $\DD$-pair bisets.
\end{proposition}
\begin{proof}
	Let
	\[
	\Id_{(L,u)}:=L
	\]
	with its usual regular $(L,L)$-biset structure and with the action of $\langle(u,u)\rangle$ given by
	\[
	(u^i,u^i)\cdot l=u^i(l).
	\]
	The unique $(L,L)$-orbit is clearly invariant under $(u,u)$, so $\Id_{(L,u)}$ is an $\big((L,u),(L,u)\big)$-biset.
	
	We now show that $\Id_{L,u}$ acts as a right identity. Let $\Omega$ be a 
	$((M,v),(L,u))$-biset. By definition, $\Omega\times_{L,u}\Id_{(L,u)}$ is the $(M,L)$-subbiset of $\Omega\times_LL$ 
	generated by the image of the map 
	$\theta: \Omega^{(v,u)}\times_{L^u}\Id_L^{(u,u)}\to \Omega\times_LL$ sending $(\omega,_{L^u}l)$ to $(\omega,_Ll)$, for 
	$\omega\in \Omega$ and $l\in \Id_L^{(u,u)}=L^u$.
%	\[
%	X\times_{L,u}\Id_{(L,u)}=M\bigl(X^{(v,u)}\times_{L^u}L^{(u,u)}\bigr)L=M\bigl(X^{(v,u)}\times_{L^u}L^u\bigr)L
%	\]
	The map
	\[
	\Omega^{(v,u)}\times_{L^u}L^u\longrightarrow \Omega, \qquad (\omega,{}_{L^u}l)\longmapsto \omega l,
	\]
	identifies \(\Omega^{(v,u)}\times_{L^u}L^u\) with \(\Omega^{(v,u)}\). Therefore
	\[
	X\times_{L,u}\Id_{(L,u)}=M\Omega^{(v,u)}L.
	\]
	Since $X$ is an $((M,v), (L,u))$-biset, every $(M\times L)$-orbit of $\Omega$ contains a point fixed by \((v,u)\), by Glauberman's Lemma. Hence
	\[
	M\Omega^{(v,u)}L=\Omega.
	\]
	This prove the claim. One can similarly show that $\Id_{(L,u)}$ is also a left identity.
\end{proof}

%%%%%%%%%%%%%%%%%%%%% SECTION 5 %%%%%%%%%%%%%%%%%%%%%%%%%%%%

\section{Functors associated to $\big((M,v),(L,u)\big)$-bisets}
Let $(M,v)$ and $(L,u)$ be $\DD$-pairs, and $\Omega$ be an $\big((M,v),(L,u)\big)$-biset. By Remark~\ref{Un 1}, if $\Phi$ is an $(L,u)$-set, we can view $\Phi$ as an $\big((L,u),(\un,1)\big)$-biset. Then $\Omega\times_{L,u}\Phi$ is an $\big((M,v),(\un,1)\big)$-biset, hence an $(M,v)$-set. By Lemma~\ref{composition is functorial}, this construction $\Phi\mapsto \Omega\times_{L,u}\Phi$ is in fact a functor $F_\Omega$ from $\pset{(L,u)}$ to $\pset{(M,v)}$. We note that this functor preserves coproducts.\par
Moreover, if $(N,w)$ is a $\DD$-pair, and if $\Gamma$ is a $\big((N,w),(M,v)\big)$-biset, then Proposition~\ref{composition is associative} shows that the functors $F_\Gamma\circ F_\Omega$ and $F_{\Gamma\times_{M,v}\Omega}$, from $\pset{(L,u)}$ to $\pset{(N,w)}$, are isomorphic.\par
We consider five special cases of these functors $F_\Omega$, associated to what we call {\em elementary} $\DD$-pair bisets, defined as follows:

\begin{definition}
	Let $(L,u)$ be a $\DD$-pair, let $X\leq L$ be $u$-stable, and set $u':=u|_X\in\Aut(X)$. Let also $Y\unlhd L$ be $u$-stable, and let	$\bar u\in\Aut(L/Y)$ be the automorphism induced by $u$.
	
	\smallskip
	{\rm (a)} The \emph{restriction} biset
	\[
	\Res^{(L,u)}_{(X,u')}\in \lbiset{(X,u')}{(L,u)}
	\]
	is defined as the ordinary $(X,L)$-biset $L$, equipped with the action of
	$\langle(u',u)\rangle$ given by
	\[
	(u'^i,u^i)\cdot l=u^i(l)\qquad (l\in L,\ i\in\ZZ).
	\]
	The corresponding functor from $\pset{(L,u)}$ to $\pset{(X,u')}$ will also be called {\em restriction}, and denoted by  $\Res^{(L,u)}_{(X,u')}$.

	\smallskip
	{\rm (b)} The \emph{induction} biset
	\[
	\Ind^{(L,u)}_{(X,u')} \in \lbiset{(L,u)}{(X,u')}
	\]
	is defined as the ordinary $(L,X)$-biset $L$, equipped with the action of
	$\langle(u,u')\rangle$ given by
	\[
	(u^i,u'^i)\cdot l=u^i(l) \qquad (l\in L,\ i\in\ZZ).
	\]
	
	The corresponding functor from $\pset{(X,u')}$ to $\pset{(L,u)}$ will also be called {\em induction}, and denoted by  $\Ind^{(L,u)}_{(X,u')}$.

	\smallskip
	{\rm (c)} The \emph{inflation} biset
	\[
	\Inf^{(L,u)}_{(L/Y,\bar u)} \in \lbiset{(L,u)}{(L/Y, \bar u)}
	\]
	is defined as the ordinary $(L,L/Y)$-biset $L/Y$, where $L$ acts on $L/Y$ through the quotient map $L\to L/Y$. We equip it with the action of $\langle(u,\bar u)\rangle$ given by
	\[
	(u^i,\bar u^i)\cdot(lY)=u^i(l)Y.
	\]
	The corresponding functor from $\pset{(L/Y,\bar{u})}$ to $\pset{(L,u)}$ will also be called {\em inflation}, and denoted by  $\Inf^{(L,u)}_{(L/Y,\bar{u})}$.

	\smallskip
	{\rm (d)} The \emph{deflation} biset
	\[
	\Def^{(L,u)}_{(L/Y,\bar u)} \in \lbiset{(L/Y,\bar u)}{(L,u)}
	\]
	is defined as the ordinary $(L/Y,L)$-biset $L/Y$, where $L$ acts on the	right through the quotient map $L\to L/Y$. We equip it with the action of $\langle(\bar u,u)\rangle$ given by
	\[
	(\bar u^i,u^i)\cdot(lY)=u^i(l)Y.
	\]
		The corresponding functor from $\pset{(L,u)}$ to $\pset{(L/Y,\bar{u})}$ will also be called {\em deflation}, and denoted by  $\Def^{(L,u)}_{(L/Y,\bar{u})}$.

	\smallskip
	{\rm (e)} Let $(M,v)$ be another $\DD$-pair and let $f: (L,u)\longrightarrow (M,v)$ be an isomorphism of $\DD$-pairs. We define the \emph{isomorphism} biset
	\[
	\Isom_f \in \lbiset{(M,v)}{(L,u)}
	\]
	to be the ordinary $(M,L)$-biset $M$, where $M$ acts on the left by left multiplication and $L$ acts on the right via $f$. We equip it with the action of $\langle(v,u)\rangle$ given by
	\[
	(v^i,u^i)\cdot m=v^i(m)
	\qquad (m\in M,\ i\in\ZZ).
	\]
		The corresponding functor from $\pset{(L,u)}$ to $\pset{(M,v)}$ will also be called {\em isomorphism}, and denoted by $\Isom_f$.
\end{definition}

%%%%%%%%%%%%%%%%%%%%%%% SECTION 6 %%%%%%%%%%%%%%%%%%%%%%%%%%%%%%%%%

\section{The diagonal $\DD$-pair bisets}

\begin{definition}
	Let $(L,u)$ and $(M,v)$ be $\DD$-pairs. A \textit{diagonal} $((M,v),(L,u))$-biset $\Omega$ is an $((M,v),(L,u))$-biset which is \emph{bifree} as an $(M,L)$-biset, i.e., 
it is free as left $M$-set and as right $L$-set.
A morphism of diagonal $((M,v),(L,u))$-bisets is just a morphism of $(M\times L)\rtimes\langle(v,u)\rangle$-sets. We denote the resulting category by
\[
\lbiset{(M,v)}{(L,u)}^\Delta\,.
\]
It is the full subcategory of $\lbiset{(M,v)}{(L,u)}$ consisting of diagonal $((M,v),(L,u))$-bisets.
\end{definition}

\begin{remark}
There is an obvious notion of disjoint union of diagonal $((M,v),(L,u))$-bisets. 
An $((M,v),(L,u))$-biset $X$ is diagonal if and only if  each $(M\times L)\langle (v,u)\rangle$-orbit of $X$ is diagonal. A transitive diagonal $((M,v),(L,u))$-biset is just a diagonal $((M,v),(L,u))$-biset which is transitive as $(M\times L)\langle(v,u)\rangle$-set. As in Remark~\ref{rem DDelta bisets}, Glauberman's theorem applies: Every $M\times L$-orbit of a diagonal $((M,v),(L,u))$-biset $X$ has a point fixed by $(v,u)$ and the set of $(v,u)$-fixed points in a given $M\times L$-orbit form an $M^v\times L^u$-orbit.
\end{remark}

\begin{proposition}\label{prop transitive diagonal bisets}
Let $(L,u)$ and $(M,v)$ be $\DD$-pairs. Up to isomorphism, the transitive diagonal $((M,v),(L,u))$-bisets are precisely the transitive $(M\times L)\langle(v,u)\rangle$-sets of the form $(M\times L)/X$, where $X\le M\times L$ is a diagonal $(v,u)$-invariant subgroup, i.e., 
$$X=\Delta(B,\alpha,A):=\{\alpha(a),a)\mid a\in A\}\,,$$ 
for an isomorphism $\alpha\colon A\to B$ between subgroups $A\le L$ and $B\le M$. 

If $X$ and $Y$ are $(v,u)$-invariant diagonal subgroups of $M\times L$, then $(M\times L)/X$ and $(M\times L)/Y$ are isomorphic as diagonal $((M,v),(L,u))$-bisets if and only if $X$ and $Y$ are conjugate by an element in $M^v\times L^u$. In particular, the isomorphism classes of transitive $((M,v),(L,u))$-bisets are in bijection with the $(M^v\times L^u)$-conjugacy classes of $(v,u)$-invariant diagonal subgroups of $M\times L$. 
\end{proposition}
\begin{proof} 
This follows from Proposition~\ref{prop transitivebisets}, since an ordinary transitive $(M,L)$-biset is bifree if and only if its stabilizer in $M\times L$ is diagonal. 
\end{proof}

\begin{lemma}\label{diagonal injective}
	Let $(L,u)$, $(M,v)$ and $(N,w)$ be $\DD$-pairs, let $\Omega$ be a diagonal $((M,v),(L,u))$-biset and let $\Gamma$ be a diagonal $((N,w),(M,v))$-biset. Then the map
	\[
	(\gamma,{}_{M^v}\omega)\in \Gamma^{(w,v)}\times_{M^v}\Omega^{(v,u)} \longmapsto (\gamma,{}_M \omega)\in \Gamma\times_M \Omega
	\]
	is injective.
\end{lemma}

\begin{proof}
	Suppose that $\gamma,\gamma'\in \Gamma^{(w,v)}$ and $\omega,\omega'\in \Omega^{(v,u)}$ are such that
	\[
	(\gamma,{}_M \omega)=(\gamma',{}_M \omega').
	\]
	Then there exists $\mu\in M$ such that
	\[
	\gamma'=\gamma\mu^{-1} \qquad\text{and}\qquad \omega'=\mu \omega.
	\]
	Taking the image under $(w,v)$ of the first equality gives
	\[
	\gamma'=\gamma\mu^{-1}=\gamma v(\mu)^{-1},
	\]
	hence $v(\mu)=\mu$, since $\Gamma$ is right free. In other words $\mu\in M^v$. Then
	\[
	(\gamma',{}_{M^v}\omega')=(\gamma\mu^{-1},{}_{M^v}\mu \omega)=(\gamma,{}_{M^v}\omega),
	\]
	as was to be shown.
\end{proof}

In what follows, we will consider
\[
\Gamma^{(w,v)}\times_{M^v}\Omega^{(v,u)}
\]
as a subset of
\[
\Gamma\times_M \Omega
\]
under the map $\theta$ from Definition~\ref{composition}.

\begin{lemma}\label{diagonal composition}
	Let $(L,u)$, $(M,v)$ and $(N,w)$ be $\DD$-pairs, let $\Omega$ be a diagonal $((M,v),(L,u))$-biset and let $\Gamma$ be a diagonal $((N,w),(M,v))$-biset. Then the composition  $\Gamma\times_{M,v}\Omega$ is a diagonal $((N,w),(M,v))$-biset.
%	\[
	%N\bigl(Y^{(w,v)}\times_{M^v}X^{(v,u)}\bigr)L
	%\]
	%of $Y\times_M X$.
\end{lemma}
\begin{proof} All we have to do is to show that the $(N,L)$-biset $\Gamma\times_{M,v}\Omega$ is bifree. This follows from Lemma~\ref{diagonal injective}, since $\Gamma\times_{M,v}\Omega$ is a subbiset of $\Gamma\times_M\Omega$
and since $\Gamma\times_M \Omega$ is again bifree as $(N,L)$-biset.
\end{proof}

Moreover we have the following:

\begin{lemma}\label{lem fixedpoints}
	Let $(L,u)$, $(M,v)$ and $(N,w)$ be $\DD$-pairs, let $\Omega$ be a diagonal $((M,v),(L,u))$-biset and let $\Gamma$ be a diagonal $((N,w),(M,v))$-biset. Then
	\[
	(\Gamma\times_{M,v}\Omega)^{(w,u)}=\Gamma^{(w,v)}\times_{M^v}\Omega^{(v,u)}.
	\]
\end{lemma}

\begin{proof}
	By the definition of the action of \(\langle(w,u)\rangle\) on $\Gamma\times_{M,v}\Omega$, every element of
	\[
	\Gamma^{(w,v)}\times_{M^v}\Omega^{(v,u)}
	\]
	is fixed by $(w,u)$. Hence
	\[
	\Gamma^{(w,v)}\times_{M^v}\Omega^{(v,u)} \subseteq (\Gamma\times_{M,v}\Omega)^{(w,u)}\,.
	\]
	Conversely, let
	\[
	z\in (\Gamma\times_{M,v}\Omega)^{(w,u)}.
	\]
	By the definition of the composition, the element $z$ lies in the $(N,L)$-orbit of some element
	\[
	s=(\gamma,{}_{M^v}\omega)\in \Gamma^{(w,v)}\times_{M^v}\Omega^{(v,u)}.
	\]
	In particular, $s$ is fixed by $(w,u)$, and the $(N,L)$-orbit $NsL$ is stable under \((w,u)\). Since $NsL$ contains the fixed point $s$, Glauberman's Lemma implies that $(NsL)^{(w,u)}$ is transitive under $N^w\times L^u$. Therefore, there exist $n\in N^w$ and $l\in L^u$ such that
	\[
	z=nsl=n(\gamma,{}_{M^v} \omega)l=(n\gamma,{}_{M^v} \omega l)\,.
	\]
	Since $n\in N^w$, $l\in L^u$, $\gamma\in \Gamma^{(w,v)}$ and $\omega\in \Omega^{(v,u)}$, we have
	\[
	(w,v)\cdot(n\gamma)=w(n)((w,v)\cdot \gamma)=n\gamma
	\]
	and
	\[
	(v,u)\cdot(\omega l)=((v,u)\cdot \omega)u(l)=\omega l.
	\]
	Therefore
	\[
	n\gamma\in \Gamma^{(w,v)} \qquad\text{and}\qquad \omega l\in \Omega^{(v,u)}.
	\]
	Hence
	\[
	z=(n\gamma,{}_{M^v}\omega l) \in \Gamma^{(w,v)}\times_{M^v}\Omega^{(v,u)}.
	\]
	This proves the reverse inclusion, and hence the equality.
\end{proof}

\begin{proposition}
	Let $(K,t)$, $(N,w)$, $(M,v)$ and $(L,u)$ be $\DD$-pairs. Let $\Theta$ be a diagonal $((K,t),(N,w))$-biset, let $\Gamma$ be a diagonal $((N,w),(M,v))$-biset, and let $\Omega$ be a diagonal $((M,v),(L,u))$-biset. Then there is a natural isomorphism of diagonal $((K,t),(L,u))$-bisets
	\[
	(\Theta\times_{N,w}\Gamma)\times_{M,v}\Omega \cong \Theta\times_{N,w}(\Gamma\times_{M,v}\Omega).
	\]
\end{proposition}
\begin{proof}
This follows from Proposition~\ref{composition is associative} and Lemma~\ref{diagonal composition}.
\end{proof}

\begin{proposition}
	For each $\DD$-pair $(L,u)$, the regular $(L,L)$-biset $L$, equipped with the action of
	$\langle(u,u)\rangle$ given by
	\[
	(u^i,u^i)\cdot l=u^i(l),
	\]
	is a diagonal $\big((L,u),(L,u)\big)$-biset, hence an identity element for the composition of diagonal $\DD$-pair bisets.
\end{proposition}

\begin{proof}
By Proposition~\ref{identity}, all we have to show is that $\Id_{(L,u)}$ is a diagonal $((L,u),(L,u))$-biset. This follows from the fact that the underlying $(L,L)$-biset $L$ is bifree.
\end{proof}

Next we prove a formula for the composition of transitive diagonal $\DD$-pair bisets.

\begin{proposition}\label{prop compositionformulabisets}
	Let $(L,u)$, $(M,v)$ and $(N,w)$ be $\DD$-pairs. Let
	\[
	\Gamma=(N\times M)/\Delta(T,\beta,Z)
	\]
	be a transitive diagonal $((N,w),(M,v))$-biset, and let
	\[
	\Omega=(M\times L)/\Delta(Y,\alpha,X)
	\]
	be a transitive diagonal $((M,v),(L,u))$-biset,
	where $\Delta(T,\beta,Z)$ is $(w,v)$-invariant and
	$\Delta(Y,\alpha,X)$ is $(v,u)$-invariant (see Proposition~\ref{prop transitive diagonal bisets}). 
	Then
	\[
	\Gamma\times_{M,v}\Omega \cong \coprod_{m\in Z^v\backslash M^v/Y^v} (N\times L)/\Delta_m,
	\]
	where
	\[
	\Delta_m =\Delta\Bigl(\beta(Z\cap {}^mY), \beta i_m\alpha,\alpha^{-1}(Z^m\cap Y)\Bigr).
	\]
\end{proposition}

\begin{proof}
	We abbreviate $\Delta_\beta:=\Delta(T,\beta,Z)$ and $\Delta_\alpha:=\Delta(Y,\alpha,X)$. First note that the element
	\[
	\Delta_\beta:=\Delta(T,\beta,Z)\in \Gamma=(N\times M)/\Delta(T,\beta,Z)
	\]
	is fixed by $(w,v)$. By Glauberman's Lemma, we have
	\[
	\Gamma^{(w,v)}=N^w \Delta_\beta M^v.
	\]
	The stabilizer of the element $\Delta_\beta$ in \(N^w\times M^v\) is
	\[
	\Delta_\beta \cap (N^w\times M^v).
	\]
	Since $\Delta_\beta$ is $(w,v)$-invariant, the isomorphism $\beta\colon Z\to T$ restricts to an isomorphism
	\[
	\beta\colon Z^v\longrightarrow T^w.
	\]
	Therefore
	\[
	\stab_{N^w\times M^v}(\Delta_\beta) = \Delta_\beta \cap (N^w\times M^v)=\Delta(T^w,\beta,Z^v)\,.
	\]
	Similarly,
	$$ \stab_{M^v\times L^u}(\Delta_\alpha)= \Delta_\alpha\cap (M^v\times L^u) = \Delta(Y^v,\alpha, X^u)\,.$$

By Lemma~\ref{diagonal composition}, $\Gamma\times_{M,v} \Omega$ is a diagonal $((N,w),(L,u))$-biset and therefore an $(N\times L, (w,u))$-set (see Remark~\ref{rem DDelta bisets}). By Remark~\ref{rem (L,u)-sets}, the $N\times L$-orbits of $\Gamma\times_{M,v} \Omega$ are in bijection with the $N^w\times L^u$-orbits of $(\Gamma\times_{M,v} \Omega)^{(w,u)} = \Gamma^{(w,v)}\times_{M^v} \Omega^{(v,u)}\subseteq \Gamma\times_M \Omega$, see Lemma~\ref{lem fixedpoints}. 
Moreover, by \cite[Lemma~2.3.24]{Bouc2010a} and its proof, the $N^w\times L^u$-orbits of $\Gamma^{(w,v)}\times_{M^v} \Omega^{(v,u)}$ are represented by the elements
$$ (\Delta_\beta ,_{M^v} m\Delta_\alpha)\,,$$
where $m\in M^v$ runs through a set of representatives of the the double cosets
\begin{equation*}
   p_2\bigl(\stab_{N^w\times M^v}(\Delta_\beta)\bigr) \backslash M^v / p_1\bigl(\stab_{M^v\times L^u}(\Delta_\alpha)\bigr)
   = Z^v \backslash M^v / Y^v\,.
 \end{equation*}
Thus, altogether, the $N\times L$-orbits of $\Gamma\times_{M,v}\Omega$ are represented by the elements 
$$(\Delta_\beta,_M m\Delta_\alpha)\in \Gamma\times_{M,v} \Omega\,.$$
But 
\begin{align*}
   \stab_{N\times L}(\Delta_\beta ,_M m\Delta_\alpha)
   & = \stab_{N\times M}(\Delta_\beta) * \stab_{M\times L}(m\Delta_\alpha) \\
   & = \stab_{N\times M}(\Delta_\beta) * \lexp{(m,1)}{\Delta_\alpha} = \Delta_m\,.
\end{align*}
This completes the proof.
\end{proof}

%%%%%%%%%%%%%%%%%% SECTION 7 %%%%%%%%%%%%%%%%%%%%%%%%%%%%%%%%%%%%

\section{$\DD$-pair biset functors}\label{sec DDelta-pair biset functors}

\begin{notation}
{\rm (a)} Let $(L,u)$ be a $\DD$-pair. We denote by
\[
\mathcal S(L,u)=\{X\leq L\mid u(X)=X\}
\]
the set of $u$-invariant subgroups of $L$, and by $[\mathcal S(L,u)]$ a set of representatives for the $L^u$-conjugacy classes in $\mathcal S(L,u)$.

We denote by $B(L,u)$ the Grothendieck group of finite $(L,u)$-sets with respect to disjoint union. By the classification of transitive $(L,u)$-sets (Corollary~\ref{cor isomclassesoftransitive}),
$B(L,u)$ is the free abelian group with the set of isomorphism classes of transitive $(L,u)$-sets as basis. Equivalently,
\[
B(L,u)=\bigoplus_{X\in[\mathcal S(L,u)]}\ZZ[L/X].
\]

\smallskip
{\rm (b)} We similarly define the corresponding Burnside groups of $\DD$-pair bisets. Let $(L,u)$ and $(M,v)$ be $\DD$-pairs. We denote by
\[
B((M,v),(L,u))
\]
the Grothendieck group of finite $((M,v),(L,u))$-bisets; and by
\[
B^\Delta((M,v),(L,u))
\] 
the Grothendieck group of finite diagonal $((M,v),(L,u))$-bisets with respect to disjoint union. 

For a commutative ring $R$, we set
\[
R B(L,u)=R\otimes_{\ZZ} B(L,u)\,.
\]

Similarly we define $R B((M,v),(L,u))$ and $R B^\Delta((M,v),(L,u))$. 
\end{notation}

\begin{definition}
	Let $R$ be a commutative ring. We define the $R$-linear $\DD$-pair biset category $R \calP$ as follows.
	\begin{itemize}
		\item The objects are $\DD$-pairs.
		\item For two $\DD$-pairs $(L,u)$ and $(M,v)$,
		\[
		\Hom_{R \calP}((L,u),(M,v))=R B((M,v),(L,u)).
		\]
		\item Composition is induced by the composition of $\DD$-pair bisets. 
		\item The identity morphism of $(L,u)$ is $[\Id_{(L,u)}]$.
	\end{itemize}
\end{definition} 

Similarly, we define the $R$-linear diagonal $\DD$-pair biset category $R \calP^{\Delta}$.

\begin{definition}
	Let $R$ be a commutative ring. The $R$-linear diagonal $\DD$-pair biset category $R \calP^{\Delta}$ is defined as follows:
	\begin{itemize}
		\item The objects are $\DD$-pairs.
		\item For two $\DD$-pairs $(L,u)$ and $(M,v)$,
		\[
		\Hom_{R \calP^\Delta}((L,u),(M,v))=R B^\Delta((M,v),(L,u)).
		\]
		\item Composition is induced by the composition of $\DD$-pair bisets. 
		\item The identity morphism of $(L,u)$ is $[\Id_{(L,u)}]$.
	\end{itemize}
\end{definition} 

\begin{definition}
	{\rm (a)} An $R$-linear $\DD$-pair biset functor is an $R$-linear functor
	\[
	R\calP \longrightarrow \lMod{R}.
	\]
	
	\smallskip
	{\rm (b)} An $R$-linear diagonal $\DD$-pair biset functor is an $R$-linear functor
	\[
	R\calP^\Delta \longrightarrow \lMod{R}.
	\]
\end{definition}

%%%%%%%%%%%%%%%%%%%%% SECTION 8 %%%%%%%%%%%%%%%%%%%%%%%%%%%

\section{An equivalence of categories}\label{sec an equivalence}

Recall that $p$ has been a fixed prime throughout this paper. For the remainder of this paper, we further fix an algebraically closed field $k$ of characteristic $p$ and an integral domain $R$ with the property that every $p'$-integer is invertible in $R$ and that $R$ contains roots of unity of arbitrary $p'$-order. This implies that $R$ has characteristic $p$ or zero. In addition, we fix an imbedding $\iota\colon \mu(k)\to \mu(R)$ of the group $\mu(k)$ of roots of unity of $k$ into the group $\mu(R)$ of roots of unity of $R$. Such an embedding always exists. 

\smallskip
Next we recall some definitions from \cite{BoucYilmaz2020}.
For finite groups $G$ and $H$ we denote by $T^\Delta(H,G)$ the Grothendieck group  (with respect to direct sums) of $p$-permutation $(kH,kG)$-bimodules which are projective from either side, or equivalently, if considered as $k[H\times G]$-modules, whose indecomposable direct summands have
diagonal vertices in $H\times G$. Further, we set $RT^\Delta(H,G):=R\otimes_\ZZ T^\Delta(H,G)$. The diagonal $p$-permutation category $Rpp_k^\Delta$ has as objects all finite groups and as morphisms from $G$ to $H$ the set $RT^\Delta(H,G)$. Composition is induced by the tensor product of bimodules and the identity of $G$ is the bimodule $kG$. The category of $R$-linear functiors from $Rpp_k^\Delta$ to $\lMod{R}$ is denoted by $\calF^\Delta_{Rpp_k}$. It is an abelian $R$-linear category.

\begin{notation}
\smallskip
{\rm (a)} Let $G$ be a finite group. We set $G^\natural:=\Hom(G,k^\times)$. 
The embedding $\iota\colon \mu(k)\to \mu(R)$ induces an embedding 
$$G^\natural\to \Hom(G, R^\times)\,, \lambda\mapsto \tilde{\lambda}\,.$$
For $\lambda\in G^\natural$, we denote by $k_\lambda$ the corresponding one-dimensional $kG$-module and by $kG_{\lambda}$ the $(kG,kG)$-bimodule which is equal to $kG$ as a $k$-vector space and with the action 
\begin{align*}
	g\cdot x\cdot h=\lambda(h)^{-1}gxh, \quad\text{for } g,x,h\in G\,.
\end{align*}

\smallskip
{\rm (b)} Let $(L,u)$ be a $\DD$-pair and set $G:=\Lu$. For $x\in \langle u\rangle$, we set 
\begin{equation*}
	E^G_{L,u}(x)=\frac{1}{|u|}\sum_{\lambda\in \langle u\rangle^\natural}\tilde{\lambda}(x)^{-1} kG_{\mathrm{Id},\lambda}\in R\TD(G,G)
\end{equation*}
and
\begin{equation*}
   E^G_{L,u}:= E^G_{L,u}(u)\,.
\end{equation*}
\end{notation}

\begin{lemma}\label{lem idempotens}
Let $(L,u)$ be a $\DD$-pair and set $G:=\Lu$.
The elements $E^G_{L,u}(x)$, for $x\in\langle u\rangle$ are pairwise orthogonal idempotents of $R\TD(G,G)$ and their sum is equal to the identity element $kG$. In particular, the representable functor $R\TD(-,G)$ splits as a direct sum
	\begin{align*}
		\bigoplus_{x\in \langle u\rangle}R\TD(-,G)E^G_{L,u}(x)\,.
	\end{align*}
\end{lemma}
\begin{proof}
	For $x,y\in\langle u\rangle$, one has
	\begin{align*}
		E^G_{L,u}(x)\cdot E^G_{L,u}(y)&=\frac{1}{|u|^2}\sum_{\lambda,\mu \in \langle u\rangle^\natural}\tilde{\lambda}(x)^{-1} \tilde{\mu}(y)^{-1}  kG_{\lambda} kG_{\mu}\\&
		=\frac{1}{|u|^2}\sum_{\lambda,\mu \in \langle u\rangle^\natural}\tilde{\lambda}(x)^{-1} \tilde{\mu}(y)^{-1}  kG_{\lambda \mu}\,.
	\end{align*}
	Setting $\rho=\lambda \mu$, we get
	\begin{align*}
		E^G_{L,u}(x)\cdot E^G_{L,u}(y)&=\frac{1}{|u|^2}\sum_{\lambda,\rho \in \langle u\rangle^\natural}\tilde{\lambda}(x)^{-1} \tilde{\lambda}(y)\tilde{\rho}(y)^{-1}  kG_{\rho}\\&
		=\frac{1}{|u|^2}\sum_{\rho \in \langle u\rangle^\natural} \left(\sum_{\lambda\in \langle u\rangle^\natural} \tilde{\lambda}(x)^{-1} \tilde{\lambda}(y)\right)  \tilde{\rho}(y)^{-1}  kG_{\rho}\,.
	\end{align*}
	The inner sum is equal to $|u|$ if $x=y$, and equal to zero otherwise. This shows that the elements $E^G_{L,u}$, for $x\in\langle u\rangle$ are pairwise orthogonal idempotents of $R\TD(G,G)$. Now the fact that their sum is equal to the identity element $kG$ follows from the observation that for any $\lambda\in\langle u\rangle^\natural$, one has
	\begin{align*}
		\frac{1}{|u|}\sum_{x\in\langle u\rangle}\tilde{\lambda}(x)^{-1}=\begin{cases}
			1, & \text{if } \lambda=1\\
			0,& \text{otherwise }.
		\end{cases}
	\end{align*}
\end{proof}

\begin{proposition}\label{prop compactandprojective}
	Let $(L,u)$ be a $\DD$-pair. Then the functor $R\TD(-,\Lu)E_{L,u}$ is compact and projective in $\calF_{R pp_k}^\Delta$. 
\end{proposition}
\begin{proof}
	The proof is similar to the proof of Lemma~6.3 in \cite{BoucYilmaz2022}.
\end{proof}

\begin{proposition}\label{prop generatorsoffuncat}
	Every functor in $\calF_{R pp_k}^\Delta$ is a quotient of a direct sum of representable functors $R\TD(-,\Lu)E_{L,u}$.
\end{proposition}
\begin{proof}
	The proof is similar to the proof of Proposition~6.4 in \cite{BoucYilmaz2022}. We add some details for convenience. 
	
	Since representable functors generate $\calF_{R pp_k}^\Delta$, it suffices to prove the statement for representable functors $R\TD(-,G)$. By induction on $|G|$, we can assume that the essential algebra $\calE_R(G)\neq 0$, i.e., $G=\Lu$ for some $\DD$-pair $(L,u)$. 
	
	Now by Lemma~\ref{lem idempotens}, the identity element $kG$ is equal to the sum of idempotents $E^G_{L,u}(x)$ for $x\in\langle u\rangle$. If $x$ does not generate $\langle u\rangle$, then
	\begin{align*}
		E^G_{L,u}(x)=\frac{1}{|u|}\sum_{\mu\in \langle x\rangle^\natural} \tilde{\mu}(x)^{-1}\Biggl(\sum_{\substack{\lambda\in\langle u\rangle^\natural\\ \Res^{\langle u\rangle}_{\langle x\rangle}\lambda =\mu}} kG_{\lambda}\Biggr)
	\end{align*}
	vanishes in $\calE_R(G)$ by the proof of Theorem~4.9 in \cite{BoucYilmaz2026}. If $x$ generates $\langle u\rangle$, then the functors $R\TD(-,G)E^G_{L,u}(x)$ and $R\TD(-,G)E^{G}_{L,x}$ are equal since $E^G_{L,u}(x)=E^{G}_{L,x}$. The result follows.
\end{proof}

\begin{definition}
	Let $R\calD^\Delta$ denote the following category:
	\noindent\begin{itemize}
		\item The objects are $\DD$-pairs.
		\item $\Hom_{R\calD^\Delta}\left((L,u),(M,v)\right)=E_{M,v}R T^\Delta(\Mv,\Lu)E_{L,u}\,.$
		\item The composition is induced from tensor product of bimodules.
		\item $\Id_{L,u}=E_{L,u}k\Lu E_{L,u}\,.$
	\end{itemize}
\end{definition}

\begin{corollary}\label{cor equivalenceoffuncats}
	The category $\calF_{R pp_k}^\Delta$ of diagonal $p$-permutation functors over $R$ is equivalent to the category of $R$-linear functors from $R\calD^\Delta$ to $\lMod{R}$.
\end{corollary}
\begin{proof}
	This follows from Propositions~\ref{prop compactandprojective}, \ref{prop generatorsoffuncat} and \cite[Theorem~5.1]{BoucYilmaz2022}.
\end{proof}

Our goal is to prove
that the category $\calF_{Rpp_k}^\Delta$ of diagonal $p$-permutation functors over $R$ is
equivalent to the category of $R$-linear functors on a quotient of the
category $R\calP^\Delta$ of diagonal $\DD$-pair bisets over $R$.

We will prove this by constructing an ideal $\calI^\Delta$ in
$R\calP^\Delta$ and by showing that the quotient category
\[
\overline{R\calP^\Delta}:=R\calP^\Delta/\calI^\Delta
\]
is isomorphic to $R\calD^\Delta$. The remainder of the paper is devoted to 
the
construction of this isomorphism. We begin in the next section by constructing a convenient basis for
\[
E_{M,v}R T^\Delta(\Mv,\Lu)E_{L,u}.
\]

%%%%%%%%%%%%%%% SECTION 9 %%%%%%%%%%%%%%%%%%%%%%%%%%%%%

\section{A basis}

Throughout this section we fix two $\DD$-pairs $(M,v)$ and $(L,u)$. The goal of this section is to give an $R$-basis for $E_{M,v}RT^\Delta(\Mv,\Lu)E_{L,u}$.

Recall that, for a finite group $G$, we denote by $T(G)$ the trivial source ring of $G$, i.e., the Grothendieck group of $p$-permutation $kG$-modules with respect to direct sums and that $RT(G)=R\otimes_\ZZ T(G)$.

\begin{notation}
(a) For any finite group $G$, we denote by $RG^\natural$ the $R$-submodule of $RT(G)$ spanned by the elements $k_\lambda$, $\lambda\in G^\natural$. Moreover we set
\begin{equation*}
   \chi_g:=\chi_g^G:=\frac{1}{|G^\natural|} \sum_{\lambda\in G^\natural} \tilde{\lambda}(g^{-1}) k_\lambda\in RG^\natural\subseteq RT(G)\,.
\end{equation*}

\smallskip
(b) Let $Q=\Delta(Y,\alpha,X)\le M\times L$ be a diagonal subgroup which is invariant by $(v,u)$. We define
\begin{equation*}
   \Psi^{M,v}_{L,u}(Q):=
   \frac{|v|}{\gcd(|v|,|u|)}
   \Ind^{\Mv\times\Lu}_{Q\cdot \langle (v,u)\rangle} (\chi_{(v,u)}^{Q\cdot\langle (v,u)\rangle})  \in RT^\Delta(\Mv,\Lu))\,.
\end{equation*}
\end{notation}

\begin{remark}\label{rem chi}
Let $G$ be a normal Sylow $p$-subgroup $P$ with abelian quotient $G/P$. 

\smallskip
(a) Clearly $|G^\natural|=[G:P]=|G|_{p'}$ and one has an $R$-module isomorphism
\begin{equation*}
   RG^\natural \longrightarrow F_P(G,R)\,, \quad k_\lambda\mapsto\tilde{\lambda}\,,
\end{equation*}
between the $R$-scalar extension of the Grothendieck group $T(G)$ of $p$-permutation $kG$-modules and the $R$-module of functions from $G$ to $R$ which are constant on cosets $gP\in G/P$. 
It is easy to verify that under the above isomorphism the element $\chi^G_{g}$ corresponds to the characteristic function on the coset $gP$.

(b) If $S$ is a complement of $P$ in $G$, one has  has group isomorphisms $(G/P)^\natural \to G^\natural$ and $G^\natural \to S^\natural$ given by inflation and restriction. They induce $R$-module isomorphisms
\begin{equation*}
   \Inf_{G/P}^G\colon R(G/P)^\natural\to RG^\natural\quad\text{and}\quad \Res^G_S\colon RG^\natural\to RS^\natural\,.
\end{equation*}

\smallskip
(c) Let $H$ be a subgroup of $G$. Then $Q:=H\cap P$ is a normal Sylow $p$-subgroup of $H$ and $H/Q\cong HP/P$ is abelian. Moreover, the group homomorphism $G^\natural\to H^\natural$, $\lambda\mapsto \lambda|_H$, is surjective and every element $\mu\in H^\natural$ has precisely $[G:H]_{p'}=|G^\natural|/|H^\natural|$ preimages. This implies that, for all $g\in G$, one has
\begin{equation*}
   \Res^G_H(\chi^G_g)=\begin{cases} \chi^H_h\,,& \text{if $g\in PH$,}\\ 0\,, & \text{otherwise} \end{cases}
\end{equation*}
where $h\in H$ is such that $hP=gP$, and that
\begin{equation*}
   \Ind_H^G(\chi^H_h)= \Ind_H^G(\Res^G_H(\chi^G_h)) = \chi^G_h\cdot \Ind_H^G(k)\,,
\end{equation*}
for all $h\in H$. 

\smallskip
(d) It is a straightforward computation that, for any $\theta\in G^\natural$ and $g\in G$, one has
\begin{equation*}
   k_\theta\cdot \chi_g = \tilde{\theta}(g)\cdot \chi_g\,.
\end{equation*}

\smallskip
(e) Note that Parts~(a)--(c) apply to the groups $\Lu$, $\Mv$, and $\Mv\times \Lu$. In particular, if $Q=\Delta(Y,\alpha,X)\le M\times L$ is a diagonal subgroup which is invariant by $(v,u)$ then
\begin{align*}
   \Psi^{M,v}_{L,u}(Q) &= \frac{|v|}{\gcd(|v|,|u|)} \Ind^{\Mv\times\Lu}_{Q\cdot \langle (v,u)\rangle} (\chi_{(v,u)}^{Q\cdot\langle(v,u)\rangle}) \\
   & = \frac{|v|}{\gcd(|v|,|u|)} \chi_{(v,u)}^{\Mv\times \Lu}\cdot \Ind^{\Mv\times\Lu}_{Q\cdot \langle (v,u)\rangle}(k)\,.
\end{align*}
\end{remark}

\begin{lemma}\label{lem twistbycharacters}
Let $W$ be a $(k\Mv,k\Lu)$-bimodule, let $\mu\in\langle v\rangle^\natural$ and let $\lambda\in\langle u\rangle^\natural$. Let also $W_{\mu,\lambda}$ be the $(k\Mv,k\Lu)$-bimodule equal to $W$ as a $k$-vector space and with the action
\begin{align*}
b\cdot w\cdot a=\mu(b)\lambda(a)^{-1}bwa\,,
\end{align*}
for $b\in\Mv, a\in \Lu$ and $w\in W$.
Then we have an isomorphism 
\begin{align*}
\left(k\Mv\right)_\mu\otimes_{k\Mv}W\otimes_{k\Lu}\left(k\Lu\right)_\lambda\cong W_{\mu,\lambda}
\end{align*}
of $(k\Mv,k\Lu)$-bimodules.
\end{lemma}
\begin{proof}
The map 
\begin{align*}
\varphi: \left(k\Mv\right)_\mu\otimes_{k\Mv}W\otimes_{k\Lu}\left(k\Lu\right)_\lambda &\to W_{\mu,\lambda}\\
1\otimes w\otimes 1&\mapsto w
\end{align*}
is clearly an isomorphism of $k$-vector spaces. For $b\in \Mv, a\in\Lu$ and $w\in W$, one has
\begin{align*}
\varphi(b\cdot (1\otimes w\otimes 1)\cdot a) &=\varphi((b\cdot 1)\otimes w\otimes (1\cdot a))=\mu(b)\lambda(a)^{-1} \varphi((1\cdot b\otimes w\otimes a\cdot 1))\\&
=\mu(b)\lambda(a)^{-1}\varphi((1\otimes bwa\otimes 1))\\&
=\mu(b)\lambda(a)^{-1}bwa\,.
\end{align*}
This proves that the map $\varphi$ is an isomorphism of $(k\Mv,k\Lu)$-bimodules.
\end{proof}

\begin{lemma}\label{lem twistbycharactersofinduced}
Let $H\le \Mv\times\Lu$ and let $\theta\in H^\natural$. Moreover, let $\mu\in\Mv^\natural$ and $\lambda\in\Lu^\natural$. Then the $(k\Mv,k\Lu)$-bimodule
\begin{equation*}
\left(k\Mv\right)_\mu\otimes_{k\Mv} \Ind^{\Mv\times\Lu}_{H}(k_\theta) \otimes_{k\Lu}\left(k\Lu\right)_\lambda
\end{equation*}
when viewed as $k(\Mv\times\Lu)$-module is isomorphic to
\begin{equation*}
\Ind^{\Mv\times\Lu}_H(k_{\theta(\mu\times\lambda)|_H})\,.
\end{equation*}
\end{lemma}

\begin{proof}
By Lemma~\ref{lem twistbycharacters}, the bimodule
\begin{align*}
\left(k\Mv\right)_\mu\otimes_{k\Mv} \Ind^{\Mv\times\Lu}_{H} (k_\theta) \otimes_{k\Lu}\left(k\Lu\right)_\lambda
\end{align*}
when viewed as $k(\Mv\times\Lu)$-module is isomorphic to
\begin{equation*}
k_{\mu\times\lambda}\otimes_k \Ind_H^{\Mv\times\Lu}(k_\theta) \cong \Ind^{\Mv\times\Lu}_H(k_{\theta(\mu\times\lambda)|_H})\,.
\end{equation*}
\end{proof}

\begin{lemma}\label{lem theproductwithidempotents}
Let $H\le \Mv\times \Lu$ and let $\theta\in H^\natural$. Then
\begin{align*}
E_{M,v}\left(\Ind^{\Mv\times\Lu}_H (k_\theta)\right)E_{L,u}
=\begin{cases}
\tilde{\theta}(h)\cdot\Ind^{\Mv\times\Lu}_H (\chi_{h})& \text{if } (v,u)\in (M\times L) H\\
0& \text{if } (v,u)\notin (M\times L)H\,,
\end{cases}
\end{align*}
where $h\in H$ is such that $h\in (M\times L)(v,u)$.
\end{lemma}
\begin{proof}
By Lemma~\ref{lem twistbycharactersofinduced}, we have
\begin{align*}
	&E_{M,v}\left(\Ind^{\Mv\times\Lu}_{H}(k_\theta)\right)E_{L,u} \\
	= & \frac{1}{|v||u|} \sum_{\substack{\mu\in \Mv^\natural\\ \lambda\in\Mv^\natural}} \tilde{\mu}(v)^{-1}\tilde{\lambda}(u)^{-1}
	\Ind^{\Mv\times\Lu}_H(k_{\theta(\mu\times\lambda)|_U} )\\
	= & \frac{1}{|v||u|} \sum_{\substack{\mu\in \Mv^\natural\\ \lambda\in\Mv^\natural}} k_{\mu\times\lambda} \cdot
	\Ind^{\Mv\times\Lu}_H(k_{\theta} )\\
	= & \chi_{(v,u)}^{\Mv\times\Lu}\cdot \Ind^{\Mv\times\Lu}_H(k_{\theta} )\\
	= & \Ind^{\Mv\times\Lu}_H\bigl(\Res^{\Mv\times\Lu}_H(\chi_{(v,u)}^{\Mv\times\Lu})\cdot k_\theta\bigr)\,.
\end{align*}
By Remark~\ref{rem chi}(c),  $\Res^{\Mv\times\Lu}_H(\chi_{(v,u)}^{\Mv\times\Lu})=0$ if $(v,u)\notin (M\times L)H$ and equal to $\chi^H_w$ for any $w\in H$ with $w\in (M\times L)H$. Further, by Remark~\ref{rem chi}(d), we have $\chi_w^H\cdot k_\theta = \tilde{\theta}(w)\cdot \chi^H_w$. This completes the proof.
\end{proof}

\begin{theorem}\label{thm basis}
The elements 
$$\Psi^{M,v}_{L,u}(Q)\in E_{M,v}R T^\Delta(\Mv,\Lu)E_{L,u}\,,$$ 
where the subgroup $Q=\Delta(Y,\alpha,X)$ runs through a set of representatives of the $M^v\langle v\rangle\times L^u\langle u\rangle$-conjugacy classes of $(v,u)$-invariant diagonal subgroups of $\Mv\times\Lu$, form an $R$-basis for $E_{M,v} RT^\Delta(\Mv,\Lu)E_{L,u}$.
\end{theorem}

\begin{proof}
Set $G:=\Mv\times \Lu$.
Let $Q\le M\times L$ be a diagonal subgroup, let $U$ be a Hall $p'$-subgroup of $N_G(Q)$, and set $H:=QU$. By \cite[Theorem~4.1]{BoMo}, the isomorphism classes of indecomposable $p$-permutation $kG$-modules with vertex $Q$ are represented by the modules 
\begin{equation*}
   \Ind_H^G(k_\theta)\,,
\end{equation*} 
where $\theta$ runs through $H^\natural$. By Lemma~\ref{lem theproductwithidempotents}, the product
\begin{equation*}
   E_{M,v}\bigl(\Ind_H^G(k_\theta)\bigr) E_{L,u}
\end{equation*}
equals $0$ unless $(v,u)\in (M\times L)H=(M\times L)U$. After conjugating $Q$ with an element in $(M\times L)U$ if necessary, we may assume that $(v,u)\in U$. Thus, if none of the $G$-conjugates of $Q$ is $(v,u)$-invariant, then $E_{M,v} W E_{L,u}=0$ for all indecomposable $p$-permutation $kG$-modules $W$ with vertex $Q$.

\smallskip
Assume now that $Q$ is $(v,u)$-invariant and $H$ is chosen such that $(v,u)\in H$. Then, by Lemma~\ref{lem theproductwithidempotents} and Remark~\ref{rem chi}(c), we have
\begin{align*}
   E_{M,v}\bigl( \Ind_H^G (k_\theta) \bigr) E_{L,u} & = \tilde{\theta}(v,u)\cdot \Ind_H^G(\chi^H_{v,u})\\
   & = \tilde{\theta}(v,u)\cdot \Ind_{Q\langle(v,u)\rangle}^G(\chi^{Q\langle (v,u)\rangle}_{v,u})\,.
\end{align*}
Note that this element depends up to the scalar $\tilde{\theta}(v,u)$ only on $Q$ and not on $\theta$.

Let $\calS$ be the set of diagonal subgroups of $M\times L$ and let $\calS'$ denote the set of diagonal subgroups of $M\times L$ which have a $G$-conjugate that is $(v,u)$-invariant. So far, we have proved that the elements
\begin{equation*}
   \Ind_{Q\langle (v,u)\rangle}^G(\chi^{Q\langle (v,u)\rangle}_{(v,u)})\,,
\end{equation*}
where $Q$ runs through a set of representatives of the $G$-conjugacy classes of $\calS'$, generate the $R$-module $E_{M,v} RT^\Delta(\Mv,\Lu) E_{L,u}$. Let $\calS^{(v,u)}$ denote the subset of $\calS$ consisting of $(v,u)$-invariant subgroups in $\calS$ and let $[\calS^{(v,u)}]$ be a set of representative of the $M^v\times L^u$-conjugacy classes of $\calS^{(v,u)}$. Then, by Glauberman's Lemma, $[\calS^{(v,u)}]$ is also a set of representatives of the $G$-conjugacy classes of $\calS'$. Thus, the elements $\Psi_{(L,u)}^{(M,v)}(Q)$, with $Q\in[\calS^{(v,u)}]$ generate the $R$-module $E_{M,v} RT^\Delta(\Mv,\Lu) E_{L,u}$.

\smallskip
Now, suppose that
\[
\sum_{Q\in[\calS^{(v,u)}]} a_Q\Psi^{M,v}_{L,u}(Q)=0
\]
with $a_Q\in R$,
 and suppose that not all the coefficients $a_Q$ are zero.
 Choose $Q_0\in[\calS^{(v,u)}]$ such that $a_{Q_0}\neq 0$ and
$|Q_0|$ is maximal among all $Q\in[\calS^{(v,u)}]$ with $a_Q\neq 0$.

We apply the 
$R$-linear extension 
of the Brauer quotient at $Q_0$ to the above equation. For $Q\in[\calS^{(v,u)}]$, the element
$\Psi^{M,v}_{L,u}(Q)$ is a non-zero scalar multiple of
\[
\Ind^G_{Q\langle(v,u)\rangle}(\chi_{v,u})\,,
\]
which is relatively $Q$-projective.
Hence its Brauer quotient at $Q_0$ can be non-zero only if $Q_0$ is
$G$-conjugate to a subgroup of $Q$. By the maximality of $|Q_0|$, this can happen only if $Q_0$ and $Q$ are $G$-conjugate. Since $[\calS^{v,u}]$ has been chosen as a set of representatives of the $G$-conjugacy classes of $\calS'$, this forces $Q=Q_0$. Thus
\[
\Br_{Q_0}\bigl(\Psi^{M,v}_{L,u}(Q)\bigr)=0
\qquad\text{for all }Q\in[\calS^{v,u}],\; Q\neq Q_0.
\]
On the other hand, again by \cite[Theorem~4.1]{BoMo}, one has
\[
\Br_{Q_0}\bigl(\Psi^{M,v}_{L,u}(Q_0)\bigr)
= \kappa \cdot \Ind^{N_{G}(Q_0)}_{Q_0\langle(v,u)\rangle} (\chi_{v,u}) \neq 0\,,
\]
with a non-zero constant $\kappa\in R$.
This implies that $a_{Q_0}=0$, a contradiction. Therefore all coefficients $a_Q$ are zero, and the elements
\[
\Psi^{M,v}_{L,u}(Q),\qquad Q\in[\calS^{v,u}],
\]
are $R$-linearly independent.
\end{proof}

%%%%%%%%%%%%%%%%% SECTION 10 %%%%%%%%%%%%%%%%%%%%%%%%%%%%%%%%%%%%%%%%%%

\section{Composition formula}

The main goal of this section is a formula (see Theorem~\ref{thm  compositionformula}) that expresses the composition of two basis elements from the last section again in terms of this basis.

\begin{lemma}\label{lem gcdformula}
Let $a,b,c\in \NN$. One has
\begin{align*}
\gcd\left(\frac{b}{\gcd(a,b)},\frac{b}{\gcd(c,b)}\right)=\frac{b}{\lcm\left(\gcd(a,b),\gcd(c,b)\right)}\,.
\end{align*}
\end{lemma}
\begin{proof}
This follows from the fact that the map $i\mapsto b/i$ is an isomorphism from the lattice of divisors of $b$ (ordered by divisibility) to its opposite lattice.
\end{proof}

For the definition of the notation used in the following Lemma we refer the reader to \cite[Notations~2.3.19, 2.3.21]{Bouc2010a}

\begin{lemma}\label{lem starproduct}
Let $(N,w), (M,v)$ and $(L,u)$ be $\DD$-pairs, let $\Delta(T,\beta,Z)\le N\times M$ be a $(w,v)$-invariant subgroup and $\Delta(Y,\alpha,X)\le M\times L$ be a $(v,u)$-invariant subgroup. For any $m\in M$, set
\begin{align*}
K_m:=k_2\bigl(\Delta(T,\beta,Z)\cdot \langle (w,v)\rangle\bigr)\cap \lexp{m}{k_1\bigl(\Delta(Y,\alpha,X)\cdot \langle (v,u)\rangle\bigr)}\le \Mv\,
\end{align*}
and
\begin{align*}
\Delta(T,*,X)_m:=\left(\Delta(T,\beta,Z)\cdot \langle (w,v)\rangle\right)*\lexp{(m,1)}{\left(\Delta(Y,\alpha,X)\cdot \langle (v,u)\rangle\right)}\le \Nw\times \Lu\,.
\end{align*}

\smallskip
{\rm (a)} One has
$$ K_m=\{v^j\in C_{\langle v\rangle}(m) \mid \exists i\in\ZZ \colon v^i=v^j, w^j=1, u^i=1\}\,,$$
and if $m\in M^v$, then
\begin{align*}
|K_m|=\frac{|v|\cdot \gcd(|w|,|v|,|u|)}{\gcd(|w|,|v|)\cdot \gcd(|v|,|u|)}\,.
\end{align*}

\smallskip
{\rm (b)} The group $\Delta_m:=\Delta(\beta(Z\cap \lexp{m}Y),\beta i_m\alpha,\alpha^{-1}(Y\cap Z^m))$ is a Sylow $p$-subgroup of $\Delta(T,*,X)_m$.

\smallskip
{\rm (c)} If $m\in M^v$, then
\begin{align*}
\Delta(T,*,X)_m=\Delta_m\cdot \left(\langle (w,v)\rangle *\langle (v,u)\rangle\right)\,.
\end{align*}

\smallskip
{\rm (d)} One has
\begin{align*}
|\langle (w,v)\rangle *\langle (v,u)\rangle|=\frac{|w|\cdot |u|}{\gcd(|w|,|v|,|u|)}\,.
\end{align*}
\end{lemma}
\begin{proof}
{\rm (a)} We have
\begin{align*}
K_m&=\{b\in\Mv\,|\, (1,b)\in\Delta(T,\beta,Z)\cdot \langle (w,v)\rangle,\, (b^m,1)\in \Delta(Y,\alpha,X)\cdot \langle (v,u)\}\\&
=\{b\in\Mv\,|\, \exists z\in Z, x\in X, i,j\in \ZZ, (1,b)\in (\beta(z),z)(w^j,v^j),(b^m,1)=(\alpha(x),x)(v^i,u^i)\}\\&
=\{v^j\,|\, w^j=1,\, (v^j)^m=v^i, \, u^i=1\}\\&
=\{v^j\,|\, w^j=1,\, [m,v^j]v^j=v^i, \, u^i=1\}\\&
= \{v^j\in C_{\langle v\rangle}(m) \mid \exists i\in\ZZ \colon v^i=v^j, w^j=1, u^i=1\}\,,
%=\{v^j=v^i\,|\, w^j=1=u^i, \, v^i\in C_{\langle v\rangle}(m)\}\,,
\end{align*}
which proves the first part. Now assume that $m\in M^v$. Then
\begin{align*}
K_m &= \{v^j\mid \exists i\in\ZZ: v^i=v^j, u^i=1, w^i=1\}  \\ 
%K_m&=\{v^j=v^i\,|\, w^j=1=u^i\}\\
&=\{v^j\,|\, w^j=1\}\cap \{v^i\,|\, u^i=1\}\\
&=\langle v^{|w|}\rangle \cap \langle v^{|u|}\rangle
\end{align*}
which together with Lemma~\ref{lem gcdformula} implies that
\begin{align*}
|K_m|&=\gcd\left(\frac{|v|}{\gcd(|w|,|v|)},\frac{|v|}{\gcd(|v|,|u|)}\right)=\frac{|v|}{\lcm\left(\gcd(|w|,|v|), \gcd(|v|,|u|)\right)}\\&
=\frac{|v|\cdot \gcd(|w|,|v|,|u|)}{\gcd(|w|,|v|)\cdot \gcd(|v|,|u|)}.
\end{align*}

\smallskip
{\rm (b)} Let $(c,a)\in (N\times L)\cap \Delta(T,*,X)_m$, 
the Sylow $p$-subgroup of $\Delta(T,*,X)_m$. 
Then there exists $b\in\Mv$ such that $(c,b)\in \left(\Delta(T,\beta,Z)\cdot \langle (w,v)\rangle\right)$ and $(b^m,a)\in \left(\Delta(Y,\alpha,X)\cdot \langle (v,u)\rangle\right)$. Thus,
\begin{align*}
(c,b)=(\beta(z),z)(w^j,v^j) \quad \hbox{and}\quad (b^m,a)=(\alpha(x),x)(v^i,u^i)
\end{align*}
for some $z\in Z$, $x\in X$ and integers $i$ and $j$. Since $c\in N$, the equation $c=\beta(z)w^j$ implies that $w^j=1$ and $c=\beta(z)$. Similarly, one has $a=x$ and $u^i=1$. Moreover,
\begin{align*}
zv^j= b = \lexp{m}{\left(\alpha(x)v^i\right)}=\lexp{m}\alpha(x)[m,v^i]v^i
\end{align*}
which implies that $v^j=v^i$. Furthermore, $w^j=1$ implies that $v^j\in k_1(\langle (w,v)\rangle)$ and in particular, $v^j$ centralizes $Z$. So, $z$ is the $p$-part of $b$ and $v^j$ is the $p'$-part of $b$. Similarly, $u^i=1$ implies that $v^i$ centralizes $Y$ and so $\lexp{m}\alpha(x)$ is the $p$-part of $b$ and $\lexp{m}v^i$ is the $p'$-part of $b$. These imply that $z=\lexp{m}\alpha(x)$ and $v^j=\lexp{m}{v^i}=v^i$. 
This shows that $(c,a)\in\Delta_m$. The converse inclusion holds trivially.

\smallskip
{\rm (c)} Let $(c,a)\in \Delta(T,*,X)_m$. Then there exists $b\in\Mv$ such that $(c,b)\in\Delta(T,\beta,Z)\langle(w,v)\rangle$ and $(b^m,a)\in \Delta(Y,\alpha,X)\langle(v,u)\rangle$. Thus, we have
\begin{align*}
(c,b)=(\beta(z)w^j,zv^j)\quad \text{and}\quad (b^m,a)=(\alpha(x)v^i,xu^i)
\end{align*}
for some $z\in Z$, $x\in X$, $i,j\in \ZZ$. Since $m$ is $\langle v\rangle$-fixed, we obtain that
\begin{align*}
b=zv^j=\lexp{m}\alpha(x)\lexp{m}v^i=\lexp{m}\alpha(x)v^i
\end{align*} 
which shows that $v^j=v^i$ and $z=\lexp{m}\alpha(x)$. This proves the claim.

\smallskip
{\rm (d)} Write $D=\langle (w,v)\rangle *\langle (v,u)\rangle$. Note that
\begin{align*}
k_1(D)&=\{w^a\in\langle w\rangle\,|\, \exists b\in\ZZ, \,v^a=v^b,\, u^b=1\}\\&
=\{w^a\in \langle w\rangle\,|\, \exists k,l\in\ZZ, \,a=k|u|+l|v|\}\\&
=\langle w^{(|v|,|u|)}\rangle\,,
\end{align*}
and hence 
\begin{align*}
|k_1(D)|=\frac{|w|}{(|w|,|v|,|u|)}\,.
\end{align*}
We also have $p_2(D)=\langle u\rangle$. Therefore,
\begin{align*}
|D|=|k_1(D)|\cdot |p_2(D)|=\frac{|w|\cdot |u|}{(|w|,|v|,|u|)}\,,
\end{align*}
as desired.
\end{proof}

\begin{theorem}\label{thm compositionformula}(Composition Formula)
Let $(N,w), (M,v)$ and $(L,u)$ be $\DD$-pairs, let $\Delta(T,\beta,Z)\le N\times M$ be a $(w,v)$-invariant subgroup and $\Delta(Y,\alpha,X)\le M\times L$ be a $(v,u)$-invariant subgroup. Then
\begin{align*}
\Psi^{N,w}_{M,v}(\Delta(T,\beta,Z))\circ\Psi^{M,v}_{L,u}(\Delta(Y,\alpha,X))=\sum_{m\in Z^v\backslash M^v/Y^v} \Psi^{N,w}_{L,u}(\Delta_m)
\end{align*}
where $\Delta_m=\Delta(\beta(Z\cap \lexp{m}Y),\beta i_m\alpha,\alpha^{-1}(Y\cap Z^m))$.
\end{theorem}

\begin{proof}
By definition, the composition
\begin{align*}
\Psi^{N,w}_{M,v}(\Delta(T,\beta,Z))\circ\Psi^{M,v}_{L,u}(\Delta(Y,\alpha,X))
\end{align*}
is equal to
\begin{align*}
  \lambda\cdot \Ind^{\Nw\times\Mv}_{\Delta(T,\beta,Z)\cdot \langle (w,v)\rangle} (\chi_{(w,v)}) \otimes_{k\Mv} 
  \Ind^{\Mv\times\Lu}_{\Delta(Y,\alpha,X)\cdot \langle (v,u)\rangle}(\chi_{(v,u)})\,,
\end{align*}
where $\lambda=\frac{|w||v|}{(|w|,|v|)(|v|,|u|)}$. By~Theorem 1.1 of~\cite{tensorbimodules}, % Bouc's formula
we have
\begin{align}\label{eqn tensor of chis}
   & \Ind^{\Nw\times\Mv}_{\Delta(T,\beta,Z)\cdot \langle (w,v)\rangle} (\chi_{(w,v)}) \otimes_{k\Mv} 
  \Ind^{\Mv\times\Lu}_{\Delta(Y,\alpha,X)\cdot \langle (v,u)\rangle}(\chi_{(v,u)}) \notag \\
  = & \sum_{m\in Z\langle v\rangle\backslash \Mv/ Y \langle v\rangle} 
  \Ind^{\Nw\times\Lu}_{\Delta(T,*,X)_m} \bigl(\chi_{(w,v)} \otimes_{K_m}\lexp{(m,1)}{\chi_{(v,u)}}\bigr)\,,
\end{align}
where we can choose $m\in M$
and where
\begin{align*}
\Delta(T,*,X)_m=\left(\Delta(T,\beta,Z)\cdot \langle (w,v)\rangle\right)*\lexp{(m,1)}{\left(\Delta(Y,\alpha,X)\cdot \langle (v,u)\rangle\right)}
\end{align*}
and
\begin{align*}
K_m=k_2\bigl(\Delta(T,\beta,Z)\cdot \langle (w,v)\rangle\bigr)\cap \lexp{m}{k_1\bigl(\Delta(Y,\alpha,X)\cdot \langle (v,u)\rangle\bigr)}\,.
\end{align*}

\smallskip
First note that the subgroups $Y$ and $Z$ are $v$-invariant and that $\langle v\rangle$ acts on the set $Z\backslash M/Y$ by $v\cdot (ZmY):= v(ZmY)=Zv(m)Y$, for $m\in M$. Moreover, if $Zm'Y=Zv(m)Y$ then $Z\langle v\rangle m' Y\langle v \rangle = Z\langle v\rangle m Y\langle v \rangle$ in 
$Z\langle v\rangle \backslash M\langle v\rangle / Y\langle v \rangle$. Therefore, the function
\begin{equation}\label{eqn double cosets 1}
   \langle v\rangle\big\backslash \bigl(Z\backslash M/Y)\bigr)\longrightarrow  Z\langle v\rangle \backslash M\langle v\rangle / Y\langle v \rangle\,,
\end{equation}
induced by $ZmY\mapsto Z\langle v\rangle m Y\langle v\rangle$ is well-defined. It is obviously surjective. Moreover, it is injective. In fact, suppose that $m,m'$ are in $M$ with $m'\in Z\langle v \rangle m Y\langle v \rangle$. Then $m'=zv^j my v^i = zv^j(m) v^j (y) v^{j+i}$ for some $i,j\in\ZZ$, $z\in Z$ and $y\in Y$. But this implies $v^{j+i}=1$ and $m'\in Zv^j(m)Y$.

\smallskip
We claim that, for any $m\in M$, the element
\begin{equation}\label{eqn crucial element}
 \chi_{(w,v)}^{\Delta(T,\beta,Z)\cdot \langle (w,v)\rangle} \otimes_{K_m}
   \lexp{(m,1)}{\chi_{(v,u)}^{\Delta(Y,\alpha,X)\cdot \langle (v,u)\rangle}} \in R T(\Delta(T,*,X)_m)
\end{equation}
from Equation~(\ref{eqn tensor of chis}) is equal to $0$ if $ZmY$ is not $v$-stable and equal to
\begin{equation*}
   \frac{1}{|K_m|} \chi_{(w,u)}^{\Delta(T,*,X)_m} 
\end{equation*}
if $m\in M^v$. This will be proved in the subsequent Lemma~\ref{lem composition}.

\smallskip
Thus, we may assume that $ZmY$ is $v$-stable. By Glauberman's Lemma applied to the action of $(Z\times Y)\langle v\rangle$ on $M$ given by $(z,y)v^i\cdot m:=zv^i(m)y^{-1}$, we obtain a bijection $Z^v\backslash M^v/ Y^v\to (Z\backslash M/Y)^v$ onto the fixed points of $v$ on $Z\backslash M/Y$. Combining this with the bijection (\ref{eqn double cosets 1}), we can replace the sum over $m\in Z\langle v\rangle\backslash \Mv/ Y \langle v\rangle$ in the first paragraph by a sum over $m\in Z^v\backslash M^v/Y^v$. Moreover, if $m\in M^v$, then by Lemma~\ref{lem starproduct}(c) we have
\begin{equation*}
\Delta(T,*,X)_m=\Delta_m\cdot \left(\langle (w,v)\rangle *\langle (v,u)\rangle\right)\,.
\end{equation*}

\smallskip
Altogether, we now have
\begin{align*}
   \Psi^{N,w}_{M,v}(\Delta(T,\beta,Z))\circ\Psi^{M,v}_{L,u}(\Delta(Y,\alpha,X))&=\sum_{m\in Z^v\backslash M^v/Y^v} \frac{\lambda}{|K_m|} \Ind^{\Nw\times\Lu}_{\Delta(T,*,X)_m} ( \chi_{(w,u)})
\end{align*}
which is equal to
\begin{align*}
\sum_{m\in Z^v\backslash M^v/Y^v} \frac{\lambda\cdot |\langle (w,u)\rangle|}{|K_m|\cdot |\langle(w,v)\rangle*\langle(v,u)\rangle|}\Ind^{\Nw\times\Lu}_{\Delta_m\cdot \langle (w,u)\rangle}  (\chi_{(w,u)})\,.
\end{align*}
By Lemma~\ref{lem starproduct}, one has
\begin{align*}
|\langle(w,v)\rangle*\langle(v,u)\rangle|=\frac{|w|\cdot |u|}{\gcd(|w|,|v|,|u|)}\quad\text{and}\quad |K_m|=\frac{|v|\cdot \gcd(|w|,|v|,|u|)}{\gcd(|w|,|v|)\cdot \gcd(|v|,|u|)}\,,
\end{align*}
and so
\begin{align*}
|K_m||(\langle(w,v)\rangle*\langle(v,u)\rangle):\langle(w,u)\rangle|&=\frac{|v|\cdot \gcd(|w|,|v|,|u|)}{\gcd(|w|,|v|)\cdot \gcd(|v|,|u|)}\cdot \frac{\gcd(|w|,|u|)}{\gcd(|w|,|v|,|u|)}\\&
=\frac{|v|\cdot \gcd(|w|,|u|)}{\gcd(|w|,|v|)\cdot \gcd(|v|,|u|)}\,.
\end{align*}
Thus,
\begin{align*}
\frac{\lambda\cdot |\langle (w,u)\rangle|}{|K_m|\cdot |\langle(w,v)\rangle*\langle(v,u)\rangle|}&=\frac{|w|\cdot |v|}{\gcd(|w|,|v|)\cdot\gcd(|v|,|u|)}\cdot \frac{\gcd(|w|,|v|)\cdot \gcd(|v|,|u|)}{|v|\cdot \gcd(|w|,|u|)}\\&
=\frac{|w|}{\gcd(|w|,|u|)}\,.
\end{align*}
Therefore,
\begin{align*}
   \Psi^{N,w}_{M,v}(\Delta(T,\beta,Z))\circ\Psi^{M,v}_{L,u}(\Delta(Y,\alpha,X))
   & = \sum_{m\in Z^v\backslash M^v/Y^v} \frac{|w|}{\gcd(|w|,|u|)}
   \Ind^{\Nw\times\Lu}_{\Delta_m\cdot \langle (w,u)\rangle}  (\chi_{(w,u)})\\
   & = \sum_{m\in Z^v\backslash M^v/Y^v} \Psi^{N,w}_{L,u}(\Delta_m)
\end{align*}
as was to be shown.
\end{proof}

\begin{lemma}\label{lem composition}
Assume the hypotheses and notation from Theorem~\ref{thm compositionformula} and let $m\in M$. Then the element
\begin{equation*}
   \chi_{w,v} \otimes_{K_m}\lexp{(m,1)}{\chi_{v,u}} \in R T(\Delta(T,*,X)_m)
\end{equation*}
is equal to $0$ if $Z m Y$ is not $v$-invariant and it is equal to $\frac{1}{|K_m|}\chi_{w,u}$ if $m\in M^v$.
\end{lemma}

\begin{proof}
We abbreviate
\begin{equation*}
   H_\alpha:=\Delta(Y,\alpha,X)\langle (v,u)\rangle\quad\text{and}\quad
   H_\beta:=\Delta(T,\beta,Z)\langle(w,v)\rangle\,.
\end{equation*}
Then, since $\chi_{(w,v)}^{H_\beta}= \Res^{\Nw\times \Mv}_{H_\beta} (\chi_{(w,v)}^{\Nw\times \Mv})$ and 
$\chi^{H_\alpha}_{(v,u)} = \Res^{\Mv\times\Lu}_{H_\alpha}(\chi_{(v,u)}^{\Mv\times\Lu})$, we have
\begin{align}\label{eqn chis}
   & |w||v|^2|u| \cdot \Bigl( \chi_{(w,v)}^{H_\beta} \otimes_{K_m} \lexp{(m,1)}{\chi^{H_\alpha}_{(v,u)}} \Bigr) \\
   = & \Biggl( \sum_{\substack{\nu\in \Nw^\natural\\ \mu'\in\Mv^\natural}} \tilde{\nu}(w)^{-1}\tilde{\mu}'(v)^{-1} k_{(\nu\times\mu')_{H_\beta}}\Biggr)
      \otimes_{K_m} \lexp{(m,1)}{\Biggl( \sum_{\substack{\mu\in\Mv^\natural\\ \lambda\in\Lu^\natural}}
      \tilde{\mu}(v)^{-1}\tilde{\lambda}(u)^{-1} k_{(\mu\times\lambda)|_{H_\alpha}} \Biggr)} \notag \\
   = & \sum_{\nu,\mu',\mu,\lambda} \tilde{\nu}(w)^{-1}\tilde{\mu}'(v)^{-1} \tilde{\mu}(v)^{-1}\tilde{\lambda}(u)^{-1}
      \Bigl( k_{(\nu\times\mu')|_{H_\beta}} \otimes_{K_m} k_{(\lexp{m}{\mu}\times\lambda)|_{\lexp{(m,1)}{H_\alpha} } }\Bigr)\,. \notag
\end{align}
Moreover, by Theorem 1.1 of~\cite{tensorbimodules},
\begin{equation*}
    k_{(\nu\times\mu')|_{H_\beta}} \otimes_{K_m} k_{(\lexp{m}{\mu}\times\lambda)|_{\lexp{(m,1)}{H_\alpha} } }
    = \begin{cases}
        0\,, & \text{if $(\mu'\cdot\lexp{m}{\mu})|_{K_m}\neq 1$,}\\
        k_{\rho_{\nu,\mu',\mu,\lambda}}\,, & \text{if $(\mu'\cdot\lexp{m}{\mu})|_{K_m} = 1$,}
     \end{cases}
\end{equation*}
where $\rho_{\nu,\mu',\mu,\lambda}\in (H_\beta * \lexp{(m,1)}{H_\alpha})^\natural = \Delta(T,*,X)_m^\natural$ is given by
\begin{equation*}
   \rho_{\nu,\mu',\mu,\lambda}(c,a) = \nu(c)\mu'(b)\lexp{m}{\mu}(b)\lambda(a)\,,
\end{equation*}
where $b\in\Mv$ is such that 
\begin{equation*}
   (c,b)\in H_\beta\quad\text{and} \quad(b^m,a)\in H_\alpha\,.
\end{equation*} 
Thus, the function 
\begin{equation*}
   f\colon \Delta(T,*,X)_m\to R
\end{equation*} 
from Remark~\ref{rem chi}(a) associated to the element in (\ref{eqn chis}) is given for $(c,a)\in\Delta(T,*,X)_m$ by
\begin{align}\label{eqn product form}
   f(c,a) & =   \sum_{\substack{\nu,\mu',\mu,\lambda\\ (\mu'\lexp{m}\mu)|_{K_m}=1} }
   \tilde{\nu}(w)^{-1} \tilde{\mu}'(v)^{-1} \tilde{\mu}(v)^{-1} \tilde{\lambda}(u)^{-1} \tilde{\nu}(c) \tilde{\mu}'(b) \tilde{\mu}(m^{-1}bm) \tilde{\lambda}(a) \notag \\
   & = \Bigl(\sum_{\nu\in \Nw^\natural} \tilde{\nu}(w^{-1}c)\Bigr) \Bigl(\sum_{\lambda\in\Lu^\natural} \tilde{\lambda}(u^{-1}a)\Bigr)
   \Bigl(\sum_{\substack{\mu,\mu'\in \Mv^\natural\\ (\mu'\cdot \lexp{m}{\mu})|_{K_m} =1}} \tilde{\mu}'(v^{-1}b) \tilde{\mu}(v^{-1}m^{-1}bm)\Bigr)
\end{align}
The first sum is non-zero if and only if $w^{-1}c\in N$, in which case it is equal to $|w|$. The second sum is non-zero if and only if $u^{-1}a\in L$, in which case it is equal to $|u|$.  Since
	\begin{equation*}
	\tilde{\mu}(v^{-1}m^{-1}bm)=\tilde{\mu}(v^{-1}b)
	\quad\text{and}\quad
	(\lexp{m}{\mu})|_{K_m}=\mu|_{K_m}\,,
	\end{equation*}
	putting $\eta=\mu'\mu$, the third sum becomes
	\begin{equation*}
	|v|\sum_{\substack{\eta\in \Mv^\natural\\ \eta|_{K_m}=1}}
	\tilde{\eta}(v^{-1}b).
	\end{equation*}
	Hence, the third sum is non-zero if and only if $v^{-1}b\in MK_m$ which holds if and only if there exists $s\in K_m$ such that $bs\in Mv$.
Thus, after replacing $b$ with $bs$, if the element in (\ref{eqn chis}) is non-zero, there exists $(c,a)\in \Delta(T,*,X)_m$ and $b\in Mv$ such that
\begin{equation*}
   c\in Nw\,, \quad a\in Lu\,, \quad (c,b)\in H_\beta\,,\quad\text{and}\quad (b^m,a)\in H_\alpha\,. 
\end{equation*}
We claim that this implies that $ZmY$ is $v$-invariant.

\smallskip
Since $(c,b)\in H_\beta$, $c\in Nw$, and $b\in Mv$, there exists $z\in Z$ such that
	\begin{equation*}
	(c,b)=(\beta(z)w,zv).
	\end{equation*}
	Similarly, since $(b^m,a)\in H_\alpha$, $b^m\in Mv$ and $a\in Lu$, there exists $x\in X$ such that
	\begin{equation*}
	(b^m,a)=(\alpha(x)v,xu).
	\end{equation*}

%\smallskip
%In fact, since $(c,b)\in H_\beta$ and $(b^m,a)\in H_\alpha$, there exist $r,s\in\ZZ$, $z\in Z$, and $x\in X$ such that
%\begin{equation*}
%   (c,b)=(\beta(z) w^s, zv^s)\quad\text{and}\quad (b^m,a)=(\alpha(x)v^r,xu^r)\,.
%\end{equation*}
%Since $c\in Nw$, $c=\beta(z)w^s$ implies $w=w^s$ and since $a\in Lu$, $a=xu^r$ implies $u^r=u$. 
Further, this implies
\begin{equation*}
   zv=b= m\alpha(x)vm^{-1} = m\alpha(x)v(m^{-1}) v\,.
\end{equation*}
Thus, $v(m)=z^{-1}m\alpha(x)\in ZmY$, as claimed.

\smallskip
We have therefore shown that if $ZmY$ is not $v$-invariant then the element in (\ref{eqn chis}) is equal to $0$, which is the first statement of the Lemma.
Now suppose that $v(m)=m$. By Lemma~\ref{lem starproduct}(c), it suffices to show that the function $f\colon \Delta(T,*,X)_m\to R$ from above maps elements in $\Delta_m(w,u)$ to $|w||v|^2|u|/|K_m|$ and vanishes on all other elements. 
We have already seen that $f$ vanish on $(c,a)$ unless $(c,a)\in (N\times L)(w,u)$, which implies $(c,a)\in \Delta(T,*,X)_m\cap (N\times L)(w,u) = \Delta_m(w,u)$ by Lemma~\ref{lem starproduct}(b) and (c). Thus, it suffices to show that $f(w,u)=|w||v|^2|u|/|K_m|$. But for $(c,a)=(w,u)$ we can choose $b=v$ and Equation~(\ref{eqn product form}) gives the desired result, since $v^{-1}m^{-1}vm=v^{-1}(m^{-1}) m\in M$ and since the last sum in (\ref{eqn product form}) has $|v|^2/|K_m|$ summands.
\end{proof}

\begin{notation}
Let $(L,u)$ be a $\DD$-pair, let $X\le L$ be $v$-invariant, and let $u'\in\Aut(L)$ be the restriction of $u$. then $(X,u')$ is a $\DD$-subpair of $(L,u)$.
We set 
\begin{equation*}
   \Res^{L,u}_{X,u'}:=\Psi^{X,u'}_{L,u} (\Delta(X))\quad\text{and}\quad \Ind_{X,u'}^{L,u}:= \Psi_{X,u'}^{L,u}(\Delta(X))\,,
\end{equation*}
where $\Delta(X):=\Delta(X,\id_X,X)$.
Moreover, if $(L,u)$ and $(M,v)$ are $\DD$-pairs and $\alpha\colon L\to M$ is an isomorphism of $\DD$-pairs then we define
\begin{equation*}
   \Isom(\alpha):=\Psi_{L,u}^{M,v}(\Delta(M,\alpha,L))\,.
\end{equation*}
\end{notation}

\begin{theorem}\label{thm elementarybisets}
Let $(L,u)$ and $(M,v)$ be $\DD$-pairs and let $\Delta(Y,\alpha,X)\le M\times L$ be a $(v,u)$-invariant diagonal subgroup. Then $X$ is $u$-invariant and $Y$ is $v$-invariant. Let $(X,u')$ and $(Y,v')$ denote the resulting $\DD$-subpairs of $(L,u)$ and $(M,v)$, see Remark~\ref{rem subpairs}(a). Then $\alpha$ is an isomorphism between the $\DD$-pairs $(X,u')$ and $(Y,v')$. Moreover, one has
\begin{align*}
   \Psi^{M,v}_{L,u}(\Delta(Y,\alpha,X)) &= \Psi^{M,v}_{Y,v'}(\Delta(Y))\circ \Psi^{Y,v'}_{X,u'}(\Delta(Y,\alpha,X))\circ \Psi^{X,u'}_{L,u}(\Delta(X))\\
   &=\Ind^{M,v}_{Y,v'}\circ\Isom(\alpha)\circ \Res^{L,u}_{X,u'}\,.
\end{align*}
\end{theorem}

\begin{proof}
The first statements are obvious and the last statement follows immediately from Theorem~\ref{thm compositionformula}. 
Indeed,
\begin{align*}
   \left(\Psi^{M,v}_{Y,v'}(\Delta(Y))\circ \Psi^{Y,v'}_{X,u'}(\Delta(Y,\alpha,X))\right)\circ \Psi^{X,u'}_{L,u}(\Delta(X))
   &=\Psi^{M,v}_{X,u'}(\Delta(Y,\alpha,X))\circ \Psi^{X,u'}_{L,u}(\Delta(X))\\
   &=\Psi^{M,v}_{L,u}(\Delta(Y,\alpha,X))\,.
\end{align*}
\end{proof}

\begin{corollary}
The category $\FF\calD^\Delta$ is generated as a preadditive category by the morphisms $\Ind$, $\Isom$ and $\Res$.
\end{corollary}
\begin{proof}
This follows from Theorems~\ref{thm basis} and \ref{thm elementarybisets}.
\end{proof}
\medskip

%%%%%%%%%%%%%%%%%%%%%%%%%%%%%%%% SECTION 11 %%%%%%%%%%%%%%%%%%%%%

\section{The main theorem}\label{sec main theorem}

We keep the hypothesis on $R$ from the beginning of Section~\ref{sec an equivalence}.

\smallskip
For $\DD$-pairs $(L,u)$ and $(M,v)$ and an $((M,v),(L,u))$-biset $\Omega$, which we view as a left $(M\times L)\langle(v,u)\rangle$-set, we define $\lexp{(v,1)}{Q}$ as the $(M\times L)\langle(v,u)\rangle)$-biset with the same underlying set $\Omega$ and an element $a\in (M\times L)\langle(v,u)\rangle$ now acting as $a^{(v,1)}$ acted before. This defines a category automorphism of $\pbiset{(M,v)}{(L,u)}$. If $\Omega=(M\times L)/Q$ is transitive with $(v,u)$-invariant $Q\le M\times L$ (see Proposition~\ref{prop transitivebisets}) then $\lexp{(v,1)}{\Omega}\cong (M\times L)/\lexp{(v,1)}{Q}$.

\begin{lemma}\label{lem conjugatesontheleft}
	Let $(L,u)$, $(M,v)$ and $(N,w)$ be $\DD$-pairs. Let
	\[
	\Gamma=(N\times M)/\Delta(T,\beta,Z)
	\]
	be a transitive diagonal $((N,w),(M,v))$-biset, and let
	\[
	\Omega=(M\times L)/\Delta(Y,\alpha,X)
	\]
	be a transitive diagonal $((M,v),(L,u))$-biset. Assume that
	$\Delta(T,\beta,Z)$ is $(w,v)$-invariant and that $\Delta(Y,\alpha,X)$ is
	$(v,u)$-invariant. Then
	\[
	\Gamma\times_{M,v}\lexp{(v,1)}\Omega \cong \lexp{(w,1)}(\Gamma\times_{M,v}\Omega)
	\]
	and
	\[
	\lexp{(w,1)}\Gamma\times_{M,v}\Omega \cong \lexp{(w,1)}(\Gamma\times_{M,v}\Omega).
	\]
\end{lemma}
\begin{proof}
	We have
	\[
	\lexp{(v,1)}\Omega\cong (M\times L)/\Delta(Y,v\alpha,X).
	\]
	By Proposition~\ref{prop compositionformulabisets},
	\[
	\Gamma\times_{M,v}\lexp{(v,1)}\Omega\cong \coprod_{m\in Z^v\backslash M^v/Y^v}(N\times L)/\Delta_m,
	\]
	where
	\[
	\Delta_m=\Delta\bigl(
	\beta(Z\cap \lexp{m}Y),
	\beta i_m v\alpha,
	\alpha^{-1}v^{-1}(Z^m\cap Y)
	\bigr).
	\]
	Since $m\in M^v$, we have $i_m\circ v=v\circ i_m$. Also, the groups
	$Z$ and $Y$ are $v$-invariant. So,
	\[
	v^{-1}(Z^m\cap Y)=Z^m\cap Y
	\]
	and
	\[
	\beta i_m v\alpha=\beta v i_m\alpha=w\beta i_m\alpha,
	\]
	since $w\beta=\beta v$. Moreover, since $Z\cap\lexp{m}Y$ is $v$-invariant and $w\beta=\beta v$, the subgroup
	\[
	\beta(Z\cap{}^mY)
	\]
	is $w$-invariant. Therefore,
	\[
	\Delta_m=
	\Delta\bigl(\beta(Z\cap\lexp{m}Y),
	\beta i_m v\alpha,
	\alpha^{-1}v^{-1}(Z^m\cap Y)
	\bigr)=\lexp{(w,1)}
	\Delta\bigl(
	\beta(Z\cap\lexp{m}Y),
	\beta i_m\alpha,
	\alpha^{-1}(Z^m\cap Y)
	\bigr).
	\]
	This proves the first isomorphism. The second one is proved similarly.
\end{proof}

\begin{notation}
	For $\DD$-pairs $(M,v)$ and $(L,u)$, let
	\[
	\calI^\Delta((M,v),(L,u))
	\]
	be the $R$-submodule of $R B^\Delta((M,v),(L,u))$ generated by the elements
	\[
	[\Omega]-[\lexp{(v,1)}{\Omega}],
	\]
	where $\Omega$ runs over the set of transitive diagonal $((M,v),(L,u))$-bisets.
	By Lemma~\ref{lem conjugatesontheleft}, these subspaces form an ideal in the category	$\FF\calP^\Delta$.
\end{notation}

\begin{definition}
	The quotient category $\overline{R\calP^\Delta}$ is defined as follows.
	Its objects are the $\DD$-pairs, and for two $\DD$-pairs
	$(L,u)$ and $(M,v)$,
	\[
	\Hom_{\overline{R\calP^\Delta}}((L,u),(M,v))=\overline{R B^\Delta}((M,v),(L,u)):=R B^\Delta((M,v),(L,u))/\calI^\Delta((M,v),(L,u)).
	\]
	The composition is induced by the composition in $\FF\calP^\Delta$.
\end{definition}

Note that $\overline{R B^\Delta}((M,v),(L,u))$ is equal to the $\langle(v,1)\rangle$-cofixed points of $RB^\Delta((M,v),(L,u))$ under the action of $\langle(v,1)\rangle$ induced by $\Omega\mapsto\lexp{(v,1)}{\Omega}$. Note that this is a permutation action with $(v,1)$-stable basis given by the transitive $((M,v),(L,u))$-bisets.

For an element $\Omega\in R B^\Delta((M,v),(L,u))$, we write $\overline{\Omega}$ to denote its image in $\overline{R B^\Delta}((M,v),(L,u))$.

\begin{lemma}\label{lem basisofthequotient}
	The set of elements
	\[
	\overline{(M\times L)/Q},
	\]
	where $Q$ runs over a set of representatives of
	$M^v\langle v\rangle\times L^u\langle u\rangle$-conjugacy classes of
	$(v,u)$-invariant diagonal subgroups of $M\times L$, is an $R$-basis of
	\[
	\overline{R B^\Delta}((M,v),(L,u))\,.
	\]
\end{lemma}

\begin{proof}
Let $\calS$ be the set of $(v,u)$-invariant diagonal subgroups of $M\times L$. Let $\calT$ denote the standard basis of $RB^\Delta((M,v),(L,u))$ consisting of isomorphism classes of transitive $((M,v),(L,u))$-bisets. By Proposition~\ref{prop transitive diagonal bisets}, the map $Q\mapsto (M\times L)/Q$ induces a bijection $(M^v\times L*u)\backslash \cal S\to \calT$. 
This bijection is $\langle(v,1)\rangle$-equivariant, since $\lexp{(v,1)}{((M\times L)/Q)} \cong (M\times L)/\lexp{(v,1)}{Q}$. Under the $\langle (v,1)\rangle$-cofixed point construction, the $\langle(v,1)\rangle$-orbits of $\calT$ parametrize a basis of $\overline{R B^\Delta}((M,v),(L,u))$. By the above, they are in bijection with the $\langle (v,1)\rangle$-orbits on $(M^v\times L^u)\backslash \calS$. Since every element in $\calS$ is already $(v,u)$-fixed, these are in bijection with the $(M^v\langle v\rangle\times L^u\langle u\rangle)$-orbits of $\calS$.
\end{proof}

\begin{proposition}\label{prop comspositioninthequotient}
	Let
	\[
	\Gamma=(N\times M)/\Delta(T,\beta,Z)
	\]
	and
	\[
	\Omega=(M\times L)/\Delta(Y,\alpha,X)
	\]
	be as in Lemma~\ref{lem conjugatesontheleft}. Then in the quotient category $\overline{R\calP^\Delta}$, we
	have
	\[
	\overline \Gamma\circ \overline \Omega=\sum_{m\in Z^v\backslash M^v/Y^v}
	\overline{(N\times L)/\Delta_m},
	\]
	where
	\[
	\Delta_m=\Delta\bigl(\beta(Z\cap\lexp{m}Y),\beta i_m\alpha,\alpha^{-1}(Z^m\cap Y)\bigr).
	\]
\end{proposition}

\begin{proof}
	The composition in $\overline{R\calP^\Delta}$ is induced by the
	composition in $R\calP^\Delta$. Hence
	\[
	\overline \Gamma\circ \overline \Omega=\overline{\Gamma\times_{M,v}\Omega}.
	\]
	The result follows now from Proposition~\ref{prop compositionformulabisets}.
\end{proof}

\begin{theorem}\label{thm isomorphismofcats}
	The $R$-linear categories $\overline{R\calP^\Delta}$ and $R\calD^\Delta$ are isomorphic.
\end{theorem}
\begin{proof}
	Given $\DD$-pairs $(M,v)$ and $(L,u)$, consider the map
	\begin{align*}
		\phi^{M,v}_{L,u}:\overline{R B^\Delta}((M,v),(L,u)) &\longrightarrow E_{M,v}R T^\Delta(M\langle v\rangle,L\langle u\rangle)E_{L,u}\\
		\overline{(M\times L)/Q}&\longmapsto \Psi^{M,v}_{L,u}(Q)
	\end{align*}
	where $Q$ runs over the representatives appearing in Lemma~\ref{lem basisofthequotient}.
	By Lemma~\ref{lem basisofthequotient} and Theorem~\ref{thm basis}, each $\phi^{M,v}_{L,u}$ is an isomorphism. By Proposition~\ref{prop comspositioninthequotient} and Theorem~\ref{thm compositionformula}, they induce an isomorphism
	\[
	\phi:\overline{R\calP^\Delta}\to R\calD^\Delta
	\]
	of $R$-linear categories.
\end{proof}

\begin{corollary}
		The category $\calF_{R pp_k}^\Delta$ of diagonal $p$-permutation functors over $R$ is equivalent to the category of $R$-linear functors from $\overline{R\calP^\Delta}$ to $\lMod{R}$.
\end{corollary}
\begin{proof}
	This follows from Corollary~\ref{cor equivalenceoffuncats} and Theorem~\ref{thm isomorphismofcats}.
\end{proof}

\noindent\textbf{Acknowledgment} \ The first and second authors wish to express their gratitude for the hospitality of the Bilkent Mathematics Department during a visit in summer 2025.

%\bibliographystyle{abbrv}
%\bibliography{BBY3}

\centerline{\rule{5ex}{.1ex}}
\begin{flushleft}
	Robert Boltje, Department of Mathematics, University of California, Santa Cruz, 95064, California, USA.\\
	{\tt boltje@ucsc.edu} \vspace{1ex}\\
	Serge Bouc, CNRS-LAMFA, Universit\'e de Picardie, 33 rue St Leu, 80039, Amiens, France.\\
	{\tt serge.bouc@u-picardie.fr}\vspace{1ex}\\
	Deniz Y\i lmaz, Department of Mathematics, Bilkent University, 06800 Ankara, Turkey.\\
	{\tt d.yilmaz@bilkent.edu.tr}
\end{flushleft}
\end{document}